\documentclass{article}
\usepackage{graphicx}
\usepackage{ulem}

\usepackage[margin=3.3cm]{geometry}
\usepackage{moreverb}

\usepackage[colorlinks,bookmarksopen,bookmarksnumbered,citecolor=red,urlcolor=red]{hyperref}

\makeatletter
\newif\if@restonecol
\makeatother

\usepackage[algo2e,ruled,lined,titlenumbered,commentsnumbered]{algorithm2e}

\usepackage{tabularray}
\usepackage{siunitx}
\usepackage{makecell}
\usepackage{amsthm}
\usepackage{amsmath}
\usepackage{amssymb}
\usepackage{tikz,pgfplots}
\usepackage{subcaption}
\usepackage{multirow}
\usepackage{enumerate}
\usepackage{enumitem}
\usepackage{stmaryrd}

\theoremstyle{plain}
\newtheorem{theorem}{Theorem}
\newtheorem{assumption}[theorem]{Assumption}

\newtheorem{lemma}[theorem]{Lemma}
\newtheorem{definition}[theorem]{Definition}

\newtheorem{remark}[theorem]{Remark} 

\newcommand{\bR}{\ensuremath{{R}}}

\newcommand{\matid}{\ensuremath{{{I}}}}
\newcommand{\bP}{\ensuremath{{P}}}
\newcommand{\bB}{\ensuremath{{B}}}
\newcommand{\bD}{\ensuremath{{D}}}
\newcommand{\bM}{\ensuremath{{M}}}

\newcommand{\bW}{\ensuremath{{W}}}
\newcommand{\bZ}{\ensuremath{{Z}}}

\newcommand{\bPi}{\ensuremath{\bP_0}}

\newcommand{\bphi}{\ensuremath{\boldsymbol{\phi}}}
\newcommand{\bANeu}{\bA_{|\Omega\s}} 
\newcommand{\bL}{\ensuremath{{L}}}
\newcommand{\KPs}{\ensuremath{\tilde \Pi\s}} 

\newcommand{\bA}{\ensuremath{{A}}}
\newcommand{\bH}{\ensuremath{{H}}}
\newcommand{\Nprime}{\ensuremath{{\hat N_c}}}
\newcommand{\Ncol}{\ensuremath{{N^c}}}

\newcommand{\bu}{\ensuremath{\mathbf{u}}}
\newcommand{\by}{\ensuremath{\mathbf{y}}}

\newcommand{\bv}{\ensuremath{\mathbf{v}}}
\newcommand{\bx}{\ensuremath{\mathbf{x}}}
\newcommand{\bb}{\ensuremath{\mathbf{b}}}

\newcommand{\bz}{\ensuremath{\mathbf{z}}}

\newcommand{\st}{\ensuremath{_{t}}}
\newcommand{\0}{\ensuremath{_{0}}}
\newcommand{\sups}{\ensuremath{^{s}}}
\newcommand{\s}{\ensuremath{_{s}}}

\newcommand{\range}{\ensuremath{\operatorname{range}}}

\newcommand{\unN}{\ensuremath{\llbracket 1, N \rrbracket}}

\usepackage{stmaryrd}

\usetikzlibrary{fit,calc}

\begin{document}

\title{GenEO spectral coarse spaces in SPD domain decomposition}

\author{%
{\sc
Nicole Spillane\thanks{Email: nicole.spillane@.polytechnique.edu \\This research received no external funding.}} \\[2pt]
{\small CNRS, CMAP, \'Ecole polytechnique, 91128 Palaiseau Cedex, France} 
}

\maketitle

\pagestyle{myheadings}
\markboth{Nicole SPILLANE}{GenEO spectral coarse spaces (Nicole SPILLANE)}

\thispagestyle{plain}
\begin{abstract}
{Two-level domain decomposition methods are preconditioned Krylov solvers. What separates one- and two-level domain decomposition methods is the presence of a coarse space in the latter. The abstract Schwarz framework is a formalism that allows to define and study a large variety of two-level methods. The objective of this article is to define, in the abstract Schwarz framework, a family of coarse spaces called the GenEO coarse spaces (for Generalized Eigenvalues in the Overlaps). 
\textcolor{black}{In detail, this work is a generalization of several methods, each of which exists for a particular choice of domain decomposition method. The article both unifies the GenEO theory and extends it to new settings. The proofs are based on an abstract Schwarz theory which now applies to coarse space corrections by projection, and has been extended to consider singular local solves.} 
 Bounds for the condition numbers of the preconditioned operators are proved that are independent of the parameters in the problem (\textit{e.g.}, any coefficients in an underlying PDE or the number of subdomains). The coarse spaces are computed by finding low- or high-frequency spaces of some well-chosen generalized eigenvalue problems in each subdomain. The abstract framework is illustrated by defining two-level Additive Schwarz, Neumann-Neumann and Inexact Schwarz preconditioners for a two-dimensional linear elasticity problem. Explicit theoretical bounds as well as numerical results are provided for this example.}
\end{abstract} 


\section{Introduction}

The problem considered is the solution of linear systems of the form
\begin{equation*}
\bA \bx = \mathbf{b}, \text{ where $\bA \in \mathbb R^{n\times n} $ is symmetric positive and definite (spd).} 
\end{equation*}

The applications to bear in mind are ones for which $\bA$ is sparse and the number $n$ of unknowns is very large. Hence, parallel solvers, and more specifically domain decomposition solvers, are studied. The purpose of the article is to provide unified definitions and theory for two-level domain decomposition methods with spectral coarse spaces. This is done in the \textit{abstract Schwarz framework} by which it is referred to the formalism presented in Chapters 2 and 3 of the book by Toselli and Widlund \cite{ToselliWidlund_book2005}. This framework provides both a way of defining two-level domain decomposition preconditioners and to prove condition number bounds that involve them. 

Having chosen a partition of the global computational domain into subdomains, one-level domain decomposition preconditioners are sums of inverses of some well-chosen local problems in each of the subdomains. Two-level methods have an extra ingredient that is the coarse space. Choosing the coarse space comes down to choosing an extra, low rank, problem that is shared between all subdomains and solved at every iteration of the Krylov subspace solver. A good choice of coarse space can have a huge, positive, effect on the convergence of the method. It is with the introduction of coarse spaces that domain decomposition methods became scalable. Indeed, the first coarse spaces already ensured that, for some problems, the condition number of the two-level preconditioned operators did not depend on the number of subdomains and only weakly on the number of elements in each subdomain (see \textit{e.g.}, \cite{FarhatRoux_IJNME91,klawonn2001feti}).

\textcolor{black}{
A consensus seems to have occurred that it is worth enlarging, even quite significantly, the coarse space if this enlargement allows to achieve robustness and scalability. One popular way of doing this is to compute the coarse space by solving generalized eigenvalue problems in the subdomains. These generalized eigenvalue problems are chosen to seek out the vectors that make convergence slow. A first group of methods was tailored to the scalar elliptic problem with a varying conductivity in the Additive Schwarz framework. Among these are the two articles \cite{2010GalvisJ_EfendievY-a0,2010GalvisJ_EfendievY-ab} on one hand, and \cite{Nataf:TDD:2010,Nataf:CSC:2011,Dolean:ATL:2011} on the other. 
 The same two groups of authors contributed, with collaborators, to the set of articles  \cite{efendiev2012robust} and \cite{2011SpillaneCR,spillane2013abstract}. This time the methods apply to a much wider range of PDEs that include the linear elasticity equations. The method in \cite{2011SpillaneCR,spillane2013abstract} is called GenEO for Generalized eigenvalues in the overlaps. 
A different version of the GenEO coarse space was proposed for FETI and BDD in \cite{SPILLANE:2013:FETI_GenEO_IJNME}. The problems there are reduced to the interfaces between subdomains but the name GenEO was kept since these interfaces, in some sense, constitute an overlap between subdomains. The family of GenEO coarse spaces has grown since with \textit{e.g.}, the contributions \cite{haferssas2017additive,MR3450068} for Optimized Schwarz and \cite{marchand2020two} in the context of boundary element methods. In this article, the coarse spaces are referred to as GenEO coarse spaces since their construction generalizes the procedure for the two original GenEO coarse spaces: \cite{spillane2013abstract,SPILLANE:2013:FETI_GenEO_IJNME}.} \textcolor{black}{To date, the most general framework for spectral coarse spaces is the article \cite{agullo2019robust}.}

The idea of solving generalized eigenvalue problems to design coarse spaces with guaranteed good convergence had in fact already been proposed, unknowingly to the authors previously mentioned. Indeed, the pioneering work \cite{mandel2007adaptive} proposes such a technique for FETI-DP and BDDC. The authors make use of a `Local Indicator of the Condition Number Bound' to fill a gap in what would be an otherwise complete proof of a condition number bound. The follow-up article \cite{mandel2013adaptive} illustrates the efficiency of the method for BDDC in a multilevel framework, and \cite{klawonn2016comparison} (by different authors) makes the proof complete in two dimensions. It must also be noted that, as early as 1999, the authors of \cite{1999BrezinaM_HebertonC_MandelJ_VanekP-aa} proposed a multigrid smoothed aggregation algorithm with an enrichment technique that includes low-frequency eigenmodes of the operator in the aggregate (which is like a subdomain). Thanks to this procedure, any convergence rate chosen $\textit{a priori}$ can be achieved. Spectral enrichment is also at the heart of the spectral algebraic multigrid method \cite{2003ChartierT_FalgoutR_HensonV_JonesJ_ManteuffelT_McCormickS_RugeJ_VassilevskiP-aa}. The field of coarse spaces based on generalized eigenproblems in subdomains has been so active that it is not realistic to list all contributions here. A very incomplete overview is \cite{gander2017shem,heinlein2019adaptive,pechstein2017unified,klawonn2015toward,zampini2016pcbddc}. A topic that is currently very  active is the development of fully algebraic methods \cite{zbMATH07846109,zbMATH07695780}. 

\textbf{Contributions of the article and outline} 
 \textcolor{black}{There are three main contributions in this article. The first is to define GenEO coarse spaces for preconditioners in the abstract Schwarz framework and provide spectral bounds for the resulting preconditioners. The complete set of assumptions is clearly stated and includes singular local solves. All three coarse space corrections: projected, additive and  hybrid are considered. In order to derive the eigenvalue bounds, a significant generalization is made to the abstract framework from \cite{ToselliWidlund_book2005} both by allowing the local solvers to be singular, and by restricting the necessary conditions to projected subspaces (see in particular Lemmas~\ref{lem:upper} and~\ref{lem:stabsplit}). This is the second contribution. Finally, examples are provided of how to apply this theory to various domain decomposition methods (with analysis and numerical results). One of these methods is Inexact Schwarz with incomplete Cholesky factorization. This has not yet been consider in the GenEO literature. This is the third contribution. Since the spectral results can appear to be quite technical, they are summarized in Table~\ref{tab:absGenEO}.  The results for particular domain decomposition methods are summarized in Table~\ref{tab:abstract-app}.}

 \textcolor{black}{The outline of the article is the following. In Section~\ref{sec:abSchwarz}, the Abstract Schwarz framework is presented with minimal assumptions and the main results are stated. In Section~\ref{sec:GenEO} the results are proved. This includes generalized versions of the technical lemmas from~\cite{ToselliWidlund_book2005}.  As an illustration, Section~\ref{sec:examples} considers a two-dimensional linear elasticity problem and presents exactly how the abstract framework applies to the Additive Schwarz, Neumann-Neumann, and inexact Schwarz preconditioners.} 

\paragraph{Notation}

\begin{itemize}
\item The abbreviations spd and spsd are used to mean symmetric positive definite and symmetric positive semi-definite. 
\item $\matid$ is the identity matrix of the conforming order that is always clear in the context;
\item 
$ \langle \bx,  \by  \rangle = \bx^\top \by, \text{ and } \|\bx\| = \langle \bx,  \bx  \rangle^{1/2}, \text{ for any } \bx,\by\in\mathbb R^m$;
\item if $\bM$ is an order $m$ spd matrix, for any  $\bx,\by\in\mathbb R^m$,  
\[
\langle \bx,  \by  \rangle\textcolor{black}{_{\bM}} = \langle \bx, \bM \by\rangle, \, \|\bx\|_\bM = \langle \bx,  \bM \bx  \rangle^{1/2}, \text{ and }\bx \perp^{\bM} \by \text{ if } \langle \bx,  \bM \by  \rangle = 0;
\]
\item if it is useful to stress that the standard Euclidean inner product is considered, the notation above is used with $\ell_2$ instead of $\bM$;
\item if $\bM$ is an order $n$ spsd matrix 
 $|\bx|_\bM = \langle \bx,  \bM \by  \rangle^{1/2}\text{ for any } \bx \in\mathbb R^m$;
\item if $\bM$ is a matrix, $\lambda(\bM)$ is any one of its eigenvalues; 
\item if $\bM$ is a matrix, $\bM^\dagger$ denotes its pseudo-inverse (also called its Moore-Penrose inverse) as defined \textit{e.g}, in \cite[Problem 7.3.P7]{hornjoh:85}.
\end{itemize}

\begin{table}
\caption{\textcolor{black}{Summary of the results from Theorems~\ref{th:main} and~\ref{th:mainadd}. (Not included are the variants (labelled \textit{(b)} and \textit{(c)}) of \eqref{eq:th2-hard}). Recall that for any matrix $\bB$, the notation $\lambda(\bB)$ refers to any of its eigenvalues. For the projected operator $\bH \bA \bPi$, the lower bounds are to be understood for any non-zero eigenvalue.}}
\label{tab:absGenEO}
\centering
All the result below are under Assumptions~\ref{ass:RsVs-Astilde-Hspd} and \ref{ass:V0}. 
\scalebox{0.8}{
  \SetTblrInner{rowsep=5pt}
\begin{tblr}{hline{1-3,Z} = {2pt}, 
  hline{4-Y} = {1pt},
  vline{1-Z} = {1-2}{solid},
  vline{1-Z} = {1pt},
  colspec={ccccc}}
  \SetCell[r=2]{c} Local contributions to $V\0$ & \SetCell[c=2]{c} Requirement &
  & \SetCell[r=2]{c} Spectral bounds & \SetCell[r=2]{c} Ref. \\
 & \makecell{Ass.~\ref{ass:stable-split-matrix} \\(or \ref{ass:stable-split-matrix-withDs})} & $\tilde \bA \s$ spd &  & Ref. \\
$\operatorname{Ker}(\tilde \bA \s)$ &   &  & $\begin{aligned}  \lambda(\bH \bA \bPi) &\leq  \frac{\Ncol}{C_\sharp} \\ \lambda(\bH_{hyb} \bA )  &\leq  \max\left(1,\frac{\Ncol}{C_\sharp}\right)   \end{aligned}$ & \eqref{eq:th1-nat} \\
$\emptyset$ &   & $\checkmark$  & $\displaystyle \lambda(\bH_{ad} \bA ) \leq \frac{\Ncol}{C_\sharp} + 1 $& \eqref{eq:th2-up} \\
$\mathcal Y_L(\tau, \tilde \bA \s, \bR\s \bA  \bR\s^\top)$ &  &    & $\begin{aligned}  \lambda(\bH \bA \bPi) &\leq  \frac{\Ncol}{\tau} \\ \lambda(\bH_{hyb} \bA )  &\leq  \max\left(1,\frac{\Ncol}{\tau}\right)   \end{aligned}$   & \eqref{eq:th1-lambdamax} \\ 
$\mathcal Y_L(\tau^{-1}, \bM \s,  \tilde \bA \s )$ & $\checkmark$ &  $\checkmark$ & $\begin{aligned}   \lambda(\bH \bA \bPi) &\geq  \frac{1}{\tau \Nprime}  \\ \lambda(\bH_{hyb} \bA )  &\geq    \min \left(1,  \frac{1}{\tau \Nprime} \right)  \\ \lambda(\bH_{ad} \bA ) &\geq   \left[\max\left(2, 1 + 2 \frac{\Ncol}{C_\sharp}\right)\max(1, \Nprime \tau)  \right]^{-1}\end{aligned}$ &  \makecell{\eqref{eq:th1-lambdamin} \\ and \\  \eqref{eq:th2-nat}} \\ 
$\operatorname{Ker}(\tilde \bA \s)+\mathcal Y_H(\tau, \tilde \bA \s , \bM \s)$ & \makecell{$\checkmark$ and \\ $\bM\s$ spd} &   & $\begin{aligned}  \lambda(\bH \bA \bPi)  &\geq   \frac{1}{\tau \Nprime}  \\    \lambda(\bH_{hyb} \bA ) &\geq \min \left(1,  \frac{1}{\tau \Nprime} \right)  \end{aligned}$ & \eqref{eq:th1-lambdamin} \\ 
\makecell{$\operatorname{Ker}(\tilde \bA \s)+  \operatorname{Ker}(\bM\s) +  $ \\$\bW\s \mathcal Y_H(\tau,  \bW\s^\top \tilde \bA \s \bW\s, \bW\s^\top \bM \s  \bW\s) $}
   & $\checkmark$ &   & \makecell{Same two bounds as in the \\ cell directly above} & \eqref{eq:th1-lambdamin} \\ 
$\mathcal Y_L(\upsilon, \tilde \bA \s, \bR\s \bA  \bR\s^\top) + \mathcal Y_L(\tau^{-1}, \bM \s,  \tilde \bA \s )$  & $\checkmark$ & $\checkmark$   & $\displaystyle   \lambda(\bH_{ad} \bA ) \geq  \left[\max\left(2, 1 + 2 \frac{\Ncol}{\upsilon}\right)\max(1, \Nprime \tau)  \right]^{-1}$ & \eqref{eq:th2-hard} \\
\end{tblr}
}

\bigskip
{\footnotesize
Summary of Notation:
\begin{itemize}
\item $\bA$: problem matrix in the linear system $\bA \bx = \bb$,
\item $V\0$: coarse space,
\item $\bPi$: coarse projector (Definition~\ref{def:precs}), 
\item $\tilde \bA \s$: local solvers in the definitions of the one-level preconditioner $\bH = \sum_{s=1}^N \bR\s^\top \tilde \bA\s^\dagger \bR\s$ (Definition~\ref{def:precs}),
\item $\bH_{hyb} = \bPi \bH \bPi^\top + \bR\0^\top (\bR\0 \bA \bR\0^\top)^{-1} \bR\0 $: two-level hybrid preconditioner (Definition~\ref{def:precs}),
\item $\bH_{ad} =  \bH + \bR\0^\top (\bR\0 \bA \bR\0^\top)^{-1} \bR\0$: two-level additive preconditioner (Definition~\ref{def:precs}),
\item $\Ncol$: coloring constant (Definition~\ref{def:color}),
\item $C_\sharp$: constant that depends on $\tilde \bA \s$ (Definition~\ref{def:Csharp}),
\item $ \bR\s \bA  \bR\s^\top$: restriction of the problem to subdomain number $s$ where $\bR\s$ is the restriction to the subdomain (see Assumption~\ref{ass:RsVs-Astilde-Hspd}), 
\item $\tau, \,\upsilon >0$: thresholds chosen by the user,
\item $\mathcal Y_L(\tau, \bM_\bA, \bM_\bB) = \operatorname{span}\{\by; \, \bM_\bA \by = \lambda \bM_\bB \by \text{ with } \lambda < \tau \} $: span of low-frequency eigenvectors  (Definition~\ref{def:YLYH}),
\item $\mathcal Y_H(\tau, \bM_\bA, \bM_\bB) = \operatorname{span}\{\by; \, \bM_\bA \by = \lambda \bM_\bB \by \text{ with } \lambda \geq \tau \}$: span of high-frequency eigenvectors (Definition~\ref{def:YLYH}),
\item $\Nprime$ and $\bM\s$: from Assumption~\ref{ass:stable-split-matrix} (or \ref{ass:stable-split-matrix-withDs}),
\item $\bW\s \in \mathbb R^{n\s \times \operatorname{rank}( \bM\s)}$: matrix whose columns form an $I$-orthonormal basis of $\range(\bM\s)$.
\end{itemize}
}
\end{table}

\section{Assumptions, GenEO coarse spaces and spectral results}
\label{sec:abSchwarz}

The problem is, for a given $\mathbf{b} \in \mathbb R^{n}$, to find $\bx$ such that
\begin{equation}
\bA \bx = \mathbf{b}, \text{ where $\bA \in \mathbb R^{n\times n} $ is spd} 
\label{eq:Ax=b}
\end{equation}

\subsection{Abstract Schwarz setting and preconditioners}
\label{sub:abslocal}

We start with the components of the one-level abstract Schwarz preconditioner.  

\begin{assumption}[Local Setting]
\label{ass:RsVs-Astilde-Hspd}
Let $n \in \mathbb N$ denote the dimension of the problem matrix $\bA$. Let $N \in \mathbb N$ denote the chosen number of subdomains and $n\s \in \mathbb N$ ($s \in \unN$) denote the cardinality of each one. It is assumed that restriction operators $\bR\s \in \mathbb R^{n\times n\s}$ and local operators $\tilde \bA\s \in  \mathbb R^{n\s \times n\s} $ have been defined for each $s \in \unN$ in such a way that: 
\begin{itemize}
\item $ \bR\s \bR\s^\top = \matid$ (\textit{i.e.,} the rows of $\bR\s$ form an orthonormal basis of $ \range \left(\bR\s^\top \right)$), 
\item $\mathbb R^n = \sum_{s=1}^N \range(\bR\s^\top)$ (\textit{i.e.,}  the set of local subspaces forms a cover of the global space), 
\item each $\tilde \bA\s$ is an spsd matrix,
\item with $\tilde \bA\s^\dagger$ denoting the pseudo-inverses of the matrices $\tilde \bA \s$, the one-level preconditioner \\  $\left(\sum_{s=1}^N \bR\s^\top \tilde \bA\s^\dagger \bR\s\right)$  (later denoted by $\bH$ -- see Definition~\ref{def:precs}) is non-singular. Hence it is symmetric positive definite.  
\end{itemize}
\end{assumption}

The coloring constant, whose definition is recalled next following \cite{ToselliWidlund_book2005}[Section 2.5.1], plays an important role in the numerical performance and the theory of domain decomposition methods. For a given $\bA$, it depends only on the choice of local subspaces. 

\begin{definition}[Coloring constant]
\label{def:color}
Let $\Ncol \in \mathbb N$ be such that there exists a set $\{ \mathcal C_j; \,  1 \leq j \leq \Ncol\}$ of subsets of $\llbracket 1, N \rrbracket$ satisfying
\[
\llbracket 1,N\rrbracket = \bigcup_{1\leq j \leq \Ncol} \mathcal C_j \text{ and } \forall j \in \llbracket 1, \Ncol \rrbracket,  \, \forall \, s,t \in \mathcal C_j, \,   s\not = t \Rightarrow  \bR\s \bA \bR\st^\top = 0. 
\]
\end{definition}

One can always choose $\Ncol = N$ but in general there are values of $\Ncol$ that are significantly smaller than the number $N$ of subdomains. The number $\Ncol$ is often referred to as the coloring constant since in can be viewed as the number of colors needed to color each subdomain in such a way that any two subdomains with the same color are $\bA$-orthogonal. 

\begin{assumption}[Splitting of $\bA$]
\label{ass:stable-split-matrix}
Assume that there exist a family of $N$ spsd matrices $\bM\s \in \mathbb R^{n\s \times n \s}$ (for $s=1,\dots,N$), and a real number $\Nprime > 0$  such that any $\bx \in \mathbb R^n$ can be decomposed as
\begin{equation}
\label{eq:stable-split-matrix}
\bx = \sum_{s=1}^N \bR\s^\top \by\s,\, \by\s \in \mathbb R^{n\s} , \text{ and } \sum_{s=1}^N \langle \by\s , \bM\s \by\s \rangle \leq \Nprime  \langle \bx, \bA \bx \rangle, \text{ for every } \bx \in \mathbb R^n.
\end{equation}
\end{assumption}
The usual way of defining the GenEO coarse spaces is slightly less general and involves a family of matrices that form a partition of unity. Here, the partition of unity matrices are not assumed to be diagonal. 

\begin{assumption}[Simplification of Assumption~\ref{ass:stable-split-matrix}]
\label{ass:stable-split-matrix-withDs}
Let $\bD\s \in \mathbb R^{n\s \times n\s}$ for $s= 1, \dots, N$ be a family of spd matrices that form a partition of unity in the sense that  $\matid = \sum_{s=1}^N \bR \s^\top \bD\s \bR \s$.
Assume that there exist a set of N spsd matrices $\bM\s \in \mathbb R^{n\s \times n \s}$ (for $s=1,\dots,N$), and a real number $\Nprime > 0$  such that
\begin{equation}
\label{eq:stable-split-matrix-withDs}
\sum_{s=1}^N \langle \bD\s \bR\s \bx ,\bM\s \bD\s \bR\s \bx \rangle \leq \Nprime  \langle \bx, \bA \bx \rangle, \text{ for every } \bx \in \mathbb R^n.
\end{equation}
Assumption~\ref{ass:stable-split-matrix-withDs} implies Assumption~\ref{ass:stable-split-matrix} by choosing the splitting $\bx = \sum_{s=1}^N \bR\s^\top \by\s$ with $\by\s = \bD\s\bR\s \bx$.
\end{assumption}

To inject a second level into the preconditioners, a coarse space and a coarse solver must be chosen. The coarse space is the central topic of this article. It will be denoted by $V\0$ and the following assumption is made. 
\begin{assumption}[Coarse Setting]
\label{ass:V0}
Given a coarse space $V\0 \subset \mathbb R^n$, a basis for $V\0$ is stored in the rows of a matrix denoted ${\bR\0}$:
\[
V\0 = \range(\bR\0^\top); \quad {\bR\0} \in \mathbb R^{\dim(V\0) \times n}. 
\]
\end{assumption}
A solver must be chosen for the coarse space. In this article we will focus on the case where the coarse solver is the exact solver on the coarse space: $(\bR\0 \bA \bR\0^\top)^{-1}$.

\begin{definition}[Abstract Schwarz preconditioners]
\label{def:precs}
Under Assumption~\ref{ass:RsVs-Astilde-Hspd}, the one-level abstract Schwarz preconditioner for linear system~\eqref{eq:Ax=b} is defined by:
\begin{equation}
\label{eq:bH}
\bH := \sum_{s=1}^N \bR\s^\top \tilde \bA\s^\dagger \bR\s.
\end{equation}

Under Assumptions~\ref{ass:RsVs-Astilde-Hspd} and~\ref{ass:V0}, the two-level hybrid Schwarz preconditioner is defined by 
\[
\bH_{hyb} := \bPi \bH \bPi^\top + \bR\0^\top (\bR\0 \bA \bR\0^\top)^{-1} \bR\0,
\]
where 
\[
\bPi := \matid - \bR\0^\top (\bR\0 \bA \bR\0^\top)^{-1} \bR\0 \bA
\]
is the $\bA$-orthogonal projection satisfying $\operatorname{Ker}(\bPi) = V\0$. 

Under Assumptions~\ref{ass:RsVs-Astilde-Hspd} and~\ref{ass:V0}, the two-level Additive Schwarz preconditioner is defined by 
\[
\bH_{ad} := \bH + \bR\0^\top (\bR\0 \bA \bR\0^\top)^{-1} \bR\0.
\]
\end{definition}

\subsection{Notation}

\begin{definition}[$\mathcal Y_L(\tau, \bM_\bA, \bM_\bB)$ and $\mathcal Y_H(\tau, \bM_\bA, \bM_\bB)$]
\label{def:YLYH}
Let $m \in \mathbb N^*$, let $\bM_\bA \in \mathbb R^{m\times m}$ be an spsd matrix, let $\bM_\bB \in \mathbb R^{m\times m}$ be an spd matrix, and let  $\tau >0$ be a scalar. 
We define $\mathcal Y_L(\tau, \bM_\bA, \bM_\bB)$ and $\mathcal Y_H(\tau, \bM_\bA, \bM_\bB)$ to be the spaces of, respectively, low and high-frequency eigenvectors of the generalized eigenvalue problem $ \bM_\bA \by = \lambda \bM_\bB \by$:
\[
\mathcal Y_L(\tau, \bM_\bA, \bM_\bB) := \operatorname{span}\{\by; \, \bM_\bA \by = \lambda \bM_\bB \by \text{ with } \lambda < \tau \},
\]
and
\[
\mathcal Y_H(\tau, \bM_\bA, \bM_\bB) := \operatorname{span}\{\by; \, \bM_\bA \by = \lambda \bM_\bB \by \text{ with } \lambda \geq \tau \}.
\]
\end{definition}
There are choices of $\tau$ for which $\mathcal Y_L$ or $\mathcal Y_H$ may be empty (the space spanned by an empty set of vectors is empty).

\begin{definition}[Constant $C_\sharp$]
\label{def:Csharp}
Let $C_\sharp > 0$ be such that, for every $s=1,\dots,N$ and every $ \bx \s \in \mathbb R^{n\s}$,  
\[
\| \KPs \bR\s^\top \bx \s \|_\bA^2 \leq C_\sharp^{-1} | \bx\s|_{\tilde \bA\s}^2, 
\]
where $\KPs$ is the $\bA$-orthogonal projection characterized \footnote{If the columns in $\tilde Z\s$ form a basis for $\bR\s^\top \operatorname{Ker}(\tilde \bA \s)$ then $\KPs := \matid - \tilde \bZ\s (\tilde \bZ \s^\top \bA \tilde \bZ\s)^{-1} \tilde \bZ \s^\top \bA$} by $\operatorname{Ker}(\KPs) = \bR\s^\top\operatorname{Ker}(\tilde \bA \s)$. (If $\tilde \bA\s$ is non-singular then $\KPs$ is the identity matrix.) 
\end{definition}
The existence of such a $C_\sharp$ is clear if $\tilde \bA\s$ is non-singular: $C_\sharp$ is one of the constants in the equivalence of the norms induced by $\tilde \bA\s$ and $\bR \s\bA\bR\s^\top$. Otherwise, we notice that the terms on either side of the inequality are the semi-norms of $ \bx \s$ induced by $\bR \s{\KPs}\,^\top \bA\KPs\bR\s^\top$ and $\tilde \bA \s$. The kernel of both these operators is $\operatorname{Ker}(\tilde \bA \s)$ so they are both norms on $\range(\tilde \bA \s)$ (by \cite[Problem 5.1.P2]{hornjoh:85} and  $C_\sharp$ is one of the constants in the equivalence of these norms.

\subsection{GenEO coarse spaces and convergence results}
Recall that for any matrix $\bB$, the notation $\lambda(\bB)$ refers to any of its eigenvalues. In the next two theorems, GenEO coarse spaces are defined and bounds for the eigenvalues of the preconditioned operators are provided. These do not depend on the number of subdomains or on any parameters in $\bA$. A well-known result (see \textit{e.g.}, \cite[Lemma C.10]{ToselliWidlund_book2005}) implies that the convergence of the (projected) preconditioned conjugate gradient does not depend on these quantities either.

\begin{theorem}[Spectral results for the projected and hybrid preconditioners]
\label{th:main}
For all results in the theorem, it is assumed that Assumptions~\ref{ass:RsVs-Astilde-Hspd} and \ref{ass:V0} hold. Let $\tau >0$  be a user-chosen threshold. 
(Recall Definition~\ref{def:color} for the coloring constant $\Ncol$ and Definition~\ref{def:YLYH} for the spaces $\mathcal Y_L$ and $\mathcal Y_R$ spanned by low and high-frequency eigenvectors of a certain generalized eigenvalue problem.)

\begin{enumerate} 
\item{(Natural coarse space)} With $C_\sharp$ from Definition~\ref{def:Csharp}, if for any $s \in \unN$, $\left(\bR\s^\top \operatorname{Ker}(\tilde \bA \s)\right) \subset V\0$ then 
\begin{equation}
\label{eq:th1-nat}
\lambda(\bH \bA \bPi) \leq  \Ncol{C_\sharp}^{-1} 
\text{ and }
\lambda(\bH_{hyb} \bA )  \leq  \max\left(1,{\Ncol}{C_\sharp}^{-1}\right). 
\end{equation}
\item {(GenEO for $\lambda_{\max}$)}If for any $s \in \unN$,  $\left(\bR\s^\top \mathcal Y_L(\tau, \tilde \bA \s, \bR\s \bA  \bR\s^\top)\right)  \subset V\0 $ then 
\begin{equation}
\label{eq:th1-lambdamax}
\lambda(\bH \bA \bPi)  \leq  \Ncol {\tau}^{-1} 
\text{ and }
 \lambda(\bH_{hyb} \bA )  \leq  \max(1, \Ncol {\tau}^{-1}  ).
\end{equation}
\item {(GenEO for $\lambda_{\min}$)} Under Assumption~\ref{ass:stable-split-matrix} (or \ref{ass:stable-split-matrix-withDs}), if for any $s \in \unN$, 
\begin{enumerate}
\item either  $\tilde \bA \s$ is non-singular, and $\left(\bR\s^\top \mathcal Y_L(\tau^{-1}, \bM \s,  \tilde \bA \s )\right) \subset V\0 $,
\item \label{it:noWs} or $\bM\s$ is non-singular and $\left(\bR\s^\top 
\left[ \operatorname{Ker}(\tilde \bA \s)+ \mathcal Y_H(\tau, \tilde \bA \s , \bM \s)\right]
 \right) \subset V\0 $, 
\item \label{it:Ws} or $\bW\s \in \mathbb R^{n\s \times \operatorname{rank}( \bM\s)}$ is a matrix whose columns form an $\ell_2$-orthonormal basis of $\range(\bM\s)$
and
 \\ $\left(\bR\s^\top 
\left[ \operatorname{Ker}(\tilde \bA \s)+  \operatorname{Ker}(\bM\s) +  \bW\s \mathcal Y_H(\tau,  \bW\s^\top \tilde \bA \s \bW\s, \bW\s^\top \bM \s  \bW\s)\right]
 \right) \subset V\0 $,
\end{enumerate}
 then  
\begin{equation} 
\label{eq:th1-lambdamin}
  \left( \lambda(\bH\bA\bPi)=0 \text{ or } (\tau \Nprime)^{-1}  \leq  \lambda(\bH \bA \bPi)\right) 
\text{ and }
  \min (1,  (\tau \Nprime)^{-1} )  \leq  \lambda(\bH_{hyb} \bA ), 
\end{equation}
(with $\Nprime$ from Assumption~\ref{ass:stable-split-matrix} (or \ref{ass:stable-split-matrix-withDs})).
\end{enumerate}
\end{theorem}
\begin{proof}
By Lemma~\ref{lem:Csharp} (see below), \eqref{eq:th1-nat} is a consequence of $\eqref{eq:th1-lambdamax}$ with $\tau = C_\sharp$. We also notice that in the Assumptions for \eqref{eq:th1-lambdamin} , Case~\ref{it:noWs} is already included in case~\ref{it:Ws} (setting $\bW\s = I$). It remains to prove \eqref{eq:th1-lambdamax} and \eqref{eq:th1-lambdamin} without the assumption labelled~\ref{it:noWs}. These proofs are core results of the article. Result \eqref{eq:th1-lambdamax} for $\bH\bA\bPi$ is rewritten and proved in Theorem~\ref{th:lambdamaxHApi}.  Result \eqref{eq:th1-lambdamin} for $\bH\bA\bPi$ is rewritten and proved in Theorem~\ref{th:lambdaminHApi}. The results for $\bH_{hyb} \bA$ are then deduced by applying Theorem~\ref{th:projhyb}. 
\end{proof}
The presence of $\bW\s$ in the last case is to ensure that the matrix on the right-hand side of the generalized eigenvalue problem is spd. Faced with the same difficulty, the authors in \cite[Section 7]{MR3450068} make use of projection operators. \textcolor{black}{The next theorem is for the fully additive preconditioner. Additive coarse space correction is usually considered only for the choice $\tilde\bA\s = \bR\s \bA\bR\s^\top$ (which results in the so-called Additive Schwarz preconditioner).} 

\begin{theorem}[Spectral results for the additive preconditioner]
\label{th:mainadd}
For all results in the theorem, it is assumed that Assumptions~\ref{ass:RsVs-Astilde-Hspd} and \ref{ass:V0} hold. It is also assumed that all local operators  $\tilde \bA \s$ are non-singular (for every $s \in \unN$). Finally, let $\tau, \upsilon >0$  be two user-chosen thresholds. 
(Recall Definition~\ref{def:color} for the coloring constant $\Ncol$ and Definition~\ref{def:YLYH} for the spaces $\mathcal Y_L$ and $\mathcal Y_R$ spanned by low and high-frequency eigenvectors of a certain generalized eigenvalue problem.)

\begin{enumerate} 
\item (Natural bound for $\lambda_{\max}$) With $C_\sharp$ from Definition~\ref{def:Csharp}, it holds that
\begin{equation}
\label{eq:th2-up}
\lambda(\bH_{ad} \bA ) \leq {\Ncol}{C_\sharp^{-1}} + 1. 
\end{equation}

\item (GenEO for $\lambda_{\min}$) Under Assumption~\ref{ass:stable-split-matrix} (or \ref{ass:stable-split-matrix-withDs}), with $C_\sharp$ from Definition~\ref{def:Csharp}, if for any $s \in \unN$,  
\begin{enumerate}
\item either $\left(\bR\s^\top \mathcal Y_L(\tau^{-1}, \bM \s,  \tilde \bA \s )\right) \subset V\0 $ ,
\item or $\bM\s$ is non-singular,  and $\left(\bR\s^\top  \mathcal Y_H(\tau, \tilde \bA \s , \bM \s) \right) \subset V\0 $,
\item or 
$\bW\s \in \mathbb R^{n\s \times \operatorname{rank}( \bM\s)}$ is a matrix whose columns form an $\ell_2$-orthonormal basis of $\range(\bM\s)$
and
 $\left(\bR\s^\top 
\left[ \operatorname{Ker}(\bM\s) +  \bW\s \mathcal Y_H(\tau,  \bW\s^\top \tilde \bA \s \bW\s, \bW\s^\top \bM \s  \bW\s)\right]
 \right) \subset V\0 $,   
\end{enumerate}
then
\begin{equation}
\label{eq:th2-nat}
\left[\max\left(2, 1 + 2 {\Ncol}{C_\sharp^{-1}}\right)\max(1, \Nprime \tau)  \right]^{-1}\leq \lambda(\bH_{ad} \bA ),
\end{equation}
(with $\Nprime$ from Assumption~\ref{ass:stable-split-matrix} (or \ref{ass:stable-split-matrix-withDs})).
\item (Double GenEO for $\lambda_{\min}$) Under Assumption~\ref{ass:stable-split-matrix} (or \ref{ass:stable-split-matrix-withDs}), if for any $s \in \unN$,   
\\ $\bR\s^\top \mathcal Y_L(\upsilon, \tilde \bA \s, \bR\s \bA  \bR\s^\top) \subset V_0$
and
\begin{enumerate}
\item  either $\left(\bR\s^\top \mathcal Y_L(\tau^{-1}, \bM \s,  \tilde \bA \s )\right) \subset V\0 $ ,
\item \label{it:2noWs} or $\bM\s$ is non-singular,  and $\left(\bR\s^\top  \mathcal Y_H(\tau, \tilde \bA \s , \bM \s) \right) \subset V\0 $,
\item \label{it:2Ws} or 
$\bW\s \in \mathbb R^{n\s \times \operatorname{rank}( \bM\s)}$ is a matrix whose columns form an $\ell_2$-orthonormal basis of $\range(\bM\s)$
and
 $\left(\bR\s^\top 
\left[ \operatorname{Ker}(\bM\s) +  \bW\s \mathcal Y_H(\tau,  \bW\s^\top \tilde \bA \s \bW\s, \bW\s^\top \bM \s  \bW\s)\right]
 \right) \subset V\0 $   
\end{enumerate}
then
\begin{equation}
\label{eq:th2-hard}
\left[\max\left(2, 1 + 2 {\Ncol}{\upsilon^{-1}}\right)\max(1, \Nprime \tau)  \right]^{-1} \leq \lambda(\bH_{ad} \bA ),
\end{equation}
(with $\Nprime$ from Assumption~\ref{ass:stable-split-matrix} (or \ref{ass:stable-split-matrix-withDs})).
\end{enumerate}
\end{theorem}
\begin{proof}
It has been assumed that the $\tilde \bA\s$ are non-singular so equation~\eqref{eq:th1-nat} in Theorem~\ref{th:main} applies to the one-level preconditioner to give $\lambda(\bH \bA) \leq  \Ncol{C_\sharp}^{-1}$. 
Moreover, $\bH_{ad} \bA - \bH \bA = \bR\0^\top (\bR\0 \bA \bR\0^\top)^{-1} \bR\0 \bA$ which is a projection. Projections have eigenvalues in $\{0, 1\}$ so \eqref{eq:th2-up} holds with no restriction on the coarse space. 

By Lemma~\ref{lem:Csharp} (see below), \eqref{eq:th2-nat} is a consequence of $\eqref{eq:th2-hard}$ with $\upsilon = C_\sharp$ and non-singular $\tilde\bA\s$. We also notice that in the Assumptions for \eqref{eq:th2-hard} , Case~\ref{it:2noWs} is already included in case~\ref{it:2Ws} (setting $\bW\s = I$). It remains to prove \eqref{eq:th2-hard} without the assumption labelled~\ref{it:2noWs}. This follows from Theorem~\ref{th:add} with 
  $\min (1,  (\tau \Nprime)^{-1} )  \leq  \lambda(\bH_{hyb} \bA )$ 
 according to \eqref{eq:th1-lambdamin} in Theorem~\ref{th:main}.
\end{proof}

\subsection{Introduction to the proofs}

Subsets that are spanned by eigenvectors of a well-chosen generalized eigenvalue problem have very useful properties. Some of these are recalled for illustration and for further reference. 

\begin{lemma}
\label{lem:gevp-YLYH}
Let $m \in \mathbb N^*$, let $\bM_\bA \in \mathbb R^{m\times m}$ be an spsd matrix, let $\bM_\bB \in \mathbb R^{m\times m}$ be an spd matrix, and let $\tau >0$. With the notation from Definition~\ref{def:YLYH}, the two following properties hold
\begin{itemize}
\item spectral estimates:
\begin{equation}
\label{eq:defYLYH}
\left\{
\begin{array}{l}
|y|_{\bM_\bA}^2 < \tau \|y\|_{\bM_\bB}^2 \text{ for any } \by \in \mathcal Y_L(\tau, \bM_\bA, \bM_\bB), 
\\ 
|y|_{\bM_\bA}^2 \geq \tau \|y\|_{\bM_\bB}^2  \text{ for any } \by \in \mathcal Y_H(\tau, \bM_\bA, \bM_\bB), 
\end{array}
\right.
\end{equation}
\item conjugacy:
\begin{equation}
\label{eq:conjugacy}
\left\{
\begin{array}{l}
\left(\mathcal Y_L(\tau, \bM_\bA, \bM_\bB)\right)^{\perp^{\ell_2}} = \bM_\bB \mathcal Y_H(\tau, \bM_\bA, \bM_\bB)
\\
\left(\mathcal Y_H(\tau, \bM_\bA, \bM_\bB)\right)^{\perp^{\ell_2}} = \bM_\bB \mathcal Y_L(\tau, \bM_\bA, \bM_\bB). 
\end{array}
\right.
\end{equation}
\end{itemize}
\end{lemma}

\begin{proof}
By assumption $\bM_B$ is spd, so the generalized eigenvalue problem $\bM_A \by= \lambda \bM_B \by$ from Definition~\ref{def:YLYH} is equivalent to the classical eigenvalue problem
\[
\bM_{B}^{-1/2} \bM_A \bM_{B}^{-1/2} \bz = \lambda \bz, \text{ with } \bz = \bM_B^{1/2} \by. 
\]
By the spectral theorem \cite[Theorem 4.1.5]{hornjoh:85} applied to the spsd matrix $\bM_B^{-1/2} \bM_A \bM_B^{-1/2}$, this matrix is unitarily diagonalizable so there exists an orthonormal basis of $\mathbb R^n$ formed of eigenvectors $\bz_k$. Consequently, the vectors $\by_k = \bM_B^{-1/2} \bz_k$ form an $\bM_B$-orthonormal basis of $\mathbb R^n$ and everything else follows by analogy with the classical eigenvalue problem. 
\end{proof}

\begin{lemma}[Natural coarse space as a spectral coarse space]
\label{lem:Csharp}
Let Assumption~\ref{ass:RsVs-Astilde-Hspd} hold, then 
\[
\mathcal Y_L (C_\sharp, \tilde \bA \s, \bR\s \bA  \bR\s^\top) = \operatorname{Ker}(\tilde \bA \s), \text{ for any $s\in \unN$} 
\] 
where $C_\sharp >0$ comes from Definition~\ref{def:Csharp}.
\end{lemma}
\begin{proof}
Let $s \in \llbracket 1, N \rrbracket$. It is clear that $\operatorname{Ker}(\tilde \bA \s) \subset \mathcal Y_L (C_\sharp, \tilde \bA \s, \bR\s \bA  \bR\s^\top)$. It remains to prove that all non-zero eigenvalues $\lambda$ in the eigenvalue problem that defines $\mathcal Y_L (C_\sharp, \tilde \bA \s, \bR\s \bA  \bR\s^\top)$ are greater than $C_\sharp$. Let $\lambda \neq 0$ and $\bx \in \mathbb R^n \setminus \{\mathbf 0 \}$ satisfy 
$\tilde \bA \s \by = \lambda \bR\s \bA \bR\s^\top \by$.
 
Let $\KPs$ be as in Definition~\ref{def:Csharp} ($\bA$-orthogonal projection with $\operatorname{Ker}(\KPs) = \bR\s^\top\operatorname{Ker}(\tilde \bA \s)$). Next, we prove that  $\KPs \bR\s^\top \by = \bR\s^\top \by$ in the non-obvious case where $\tilde \bA\s$ is singular: 
\[
\range(\KPs) = \left(\operatorname{Ker}(\KPs)\right)^{\perp^\bA} = \left(\bR\s^\top \operatorname{Ker}(\tilde \bA \s)\right)^{\perp^\bA} \text { so } \bR\s^\top \by \in \range(\KPs) \Leftrightarrow \by \perp^{(\bR\s \bA \bR\s^\top)} \operatorname{Ker}(\tilde \bA\s).  
\] 
and this last assertion is true following the conjugacy property of eigenvectors (\eqref{eq:conjugacy} in Lemma~\ref{lem:gevp-YLYH}). 
We can now conclude since 
\[
 \|\by\|_{\bR\s \bA \bR\s^\top}^2= \| \bR\s^\top \by \|_\bA^2 = \| \KPs \bR\s^\top \by \|_\bA^2 \leq C_\sharp^{-1} | \by|_{\tilde \bA\s}^2 =  C_\sharp^{-1} \lambda \|\by\|_{\bR\s \bA \bR\s^\top}^2, 
\]
where the definition of $C_\sharp$ has been applied. Cancelling the common factor $\|\by\|_{\bR\s \bA \bR\s^\top}^2 \neq 0$ allows to conclude that $C_\sharp^{-1} \lambda \geq 1$. 
\end{proof}

\section{Convergence proofs for the projected preconditioner}
\label{sec:GenEO}

\subsection{Spectral bounds in the abstract framework}
The abstract Schwarz theory presented in \cite{ToselliWidlund_book2005}[Chapters 2 and 3] provides theoretical results that greatly simplify the problem of finding eigenvalue bounds for the projected preconditioned operator. For the bound on the largest eigenvalue, the results are \cite{ToselliWidlund_book2005}[Assumption 2.3, Assumption 2.4, Lemma 2.6, Lemma 2.10 and Theorem 2.13]. For the bound on the smallest eigenvalue, the results for the projected operator can be found in \cite{ToselliWidlund_book2005}[Theorem 2.13 under Assumption 2.12 (that weakens Assumption 2.2 by considering only elements in $\range(\bPi)$)]. In this section, we state and prove very similar results with the generalization that $\tilde \bA \s$ can be singular and with $\bPi$ playing a more central role in the assumptions. 

\begin{remark}
The pseudo-inverse $\bM^\dagger$ of any real matrix $\bM$ is also called the Moore—Penrose inverse of $\bM$. It is defined \textit{e.g.}, in \cite[Problem 7.3.P7]{hornjoh:85} and has the following properties that we will refer back to in the proofs involving $\tilde \bA\s^\dagger$:
\begin{equation}
\label{eq:prop-pseudoinv}
\bM^\dagger\bM\bM^\dagger=\bM^\dagger; \quad \bM\bM^\dagger\bM = \bM; \quad \text{and} \quad \range(\bM^\dagger) = \range(\bM^\top). 
\end{equation}
By symmetry, the last property implies that $\range(\tilde \bA\s^\dagger) = \range(\tilde \bA \s)$.
\end{remark}

Next, the abstract result used to bound $\lambda_{\max}$ is given and proved. Note that a difference with \cite{ToselliWidlund_book2005}[Assumption 2.4] is that the result must be proved for vectors in $\range (\tilde \bA\s^\dagger \bR\s \bPi^\top) $, instead of  $\range ( \tilde \bA\s^\dagger \bR\s) $. This subtlety is what will allow to choose the coarse space, and already appeared in \cite{SPILLANE:2013:FETI_GenEO_IJNME}[Lemma 2.8, Lemma 3.12] in the particular settings of BDD and FETI. Another difference, is the presence of a projection operator $\KPs$ in the assumption of the lemma. This weakens the assumption as long as the kernel of $\tilde \bA \s$ (once extended to the global space) is in the coarse space.    

\begin{lemma}[Upper bound for $\lambda_{\max}$]
\label{lem:upper}
Assume that the kernels of the local solvers $\tilde\bA\s$ contribute to the coarse space in the sense that 
$\sum_{s=1}^N \bR\s^\top \operatorname{Ker}(\tilde \bA \s) \subset V\0$.  
For each $s=1,\dots,N$, let $\KPs$ be the $\bA$-orthogonal projection characterized  by $\operatorname{Ker}(\KPs) = \bR\s^\top\operatorname{Ker}(\tilde \bA \s)$.  Assume that there exists $\omega > 0$ such that 
\[
\| \KPs \bR\s^\top \bx \s \|_\bA^2 \leq \omega | \bx\s|_{\tilde \bA\s}^2 \text{ for every } s = 1, \dots,N \text{ and every } \bx \s \in \range (\tilde \bA\s^\dagger \bR\s \bPi^\top).
\]
Then, the largest eigenvalue $\lambda_{\max}$ of $\bH \bA \bPi$ satisfies
 $\lambda_{\max} \leq \Ncol \omega$, 
where $\Ncol$ is as in Definition~\ref{def:color}.
\end{lemma}

\begin{proof}
Let $\bx \in \range(\bPi^\top)$. By assumption it holds that
\[
\|\KPs  \bR\s^\top \tilde \bA\s^\dagger \bR\s \bx  \|_{\bA} \leq \omega \langle \tilde \bA\s^\dagger \bR\s \bx , \tilde \bA \s \tilde \bA\s^\dagger \bR\s \bx  \rangle, \text{ for any } s = 1,\dots,N.
\]
With the notation $\bH\s := \bR\s^\top \tilde\bA\s^\dagger \bR\s$, this is equivalent to 
\begin{equation}
\label{eq:interm}
\|\KPs  \bH \s \bx \|_\bA^2 \leq \omega | \bx |_{\bH\s}^2.
\end{equation}

We next prove the intermediary result $\|\bPi \bH \bx \|_\bA^2 \leq \omega \Ncol \left\|\bx \right\|_{\bH}^2$ as follows
\begin{align*}
\|\bPi \bH \bx \|_\bA^2 &= \left\| \sum_{s=1}^N\bPi \bH\s \bx \right\|_\bA^2 
 = \left\| \sum_{j=1}^\Ncol \sum_{s\in \mathcal C_j} \bPi \bH\s \bx\right \|_\bA^2 
 \leq \left(  \sum_{j=1}^\Ncol \left\|\sum_{s\in \mathcal C_j} \bPi \bH\s \bx \right\|_\bA\right)^2 \\
& \leq \Ncol \sum_{j=1}^\Ncol \left\|\sum_{s\in \mathcal C_j} \bPi \bH\s \bx \right\|_\bA ^2
 \leq \Ncol \sum_{j=1}^\Ncol \left\|\sum_{s\in \mathcal C_j} \KPs \bH\s \bx \right\|_\bA ^2
 = \Ncol \sum_{j=1}^\Ncol \sum_{s\in \mathcal C_j} \left \|\KPs \bH\s \bx \right\|_\bA ^2 \\
 &\leq \Ncol \sum_{j=1}^\Ncol \sum_{s\in \mathcal C_j} \omega \left|\bx \right|_{\bH\s} ^2 = \Ncol \sum_{s=1}^N \omega \left|\bx \right|_{\bH\s} ^2
 = \omega \Ncol \left\|\bx \right\|_{\bH}^2,
\end{align*}
where in the first line the sets $\mathcal C_j$ are as in Definition~\ref{def:color}; in the second line the Cauchy-Schwarz estimate in the $\ell_2$-inner product, the definition of $\KPs$, as well as the definition of the sets $\mathcal C_j$ are applied; and \eqref{eq:interm} is injected into the third line. 

Next, we prove the bound for $\lambda_{\max}$ starting with the definition of an eigenvalue: 
\begin{align}
\lambda_{\max} \text{ is an eigenvalue of $\bH \bA \bPi$ } &\Leftrightarrow \lambda_{\max} \text{ is an eigenvalue of $\bPi^\top \bA \bH$ } \nonumber\\ 
&\Leftrightarrow \, \exists\, \by \in \mathbb R^n;\, \by \neq \mathbf{0} \text{ such that } \bPi^\top \bA \bH \by = \lambda_{\max} \by\label{eq:interm2}.
\end{align}
Let $\by$ be as in \eqref{eq:interm2}. It is obvious that $\by \in \range(\bPi^\top)$. Taking the inner product of \eqref{eq:interm2} by $\bH \by$, and injecting the intermediary result that was just proved gives 
\begin{equation}
\label{eq:interm3}
\lambda_{\max} \|\by\|_\bH^2 = \langle \bH \by, \bPi^\top \bA \bH \by \rangle = \|\bPi \bH \by\|_\bA^2 \leq \omega \Ncol \| \by \|_{\bH}^2.  
\end{equation}
The common factor $ \|\by\|_\bH^2$ can be cancelled since $\|\by\|_\bH^2 = 0$ would imply $\lambda_{\max} = 0$, and $\bPi = 0$. 
\end{proof}

The abstract Schwarz theory (\cite{ToselliWidlund_book2005}[Theorem 2.13 under Assumption 2.12]) also provides a result for bounding the spectrum of the two-level operator from below. The result proved in the next lemma is similar with the differences that are pointed out below the lemma. 

\begin{lemma}[Lower bound for $\lambda_{\min}$]
\label{lem:stabsplit}
Assume that the kernels of the local solvers contribute to the coarse space in the sense that
 $\sum_{s=1}^N \bR\s^\top \operatorname{Ker}(\tilde \bA \s) \subset V\0$.  
If, for any $\bx \in \operatorname{range}(\bPi)$, there exist $\bz_1, \dots, \bz_n$ such that
\[
\bx = \sum_{s=1}^N \bPi \bR\s^\top \bz \s \text{ and } \sum_{s=1}^N \langle \bz \s, \tilde \bA \s \bz \s \rangle \leq C_0^2 \langle \bx, \bA \bx \rangle \quad\text{(stable splitting of $\bx$)},
\]
then, the smallest eigenvalue $\lambda_{\min}$ of $\bH \bA \bPi$, excluding zero, satisfies
 $\lambda_{\min} \geq C_0^{-2}$.
\end{lemma}
The differences compared to \cite{ToselliWidlund_book2005}[Theorem 2.13] are the possible singularity of $\tilde \bA \s$, the extra presence of $\bPi$ in the definition of a splitting, and the extra assumption on the minimal coarse space.

\begin{proof}
Let $\bx \in \range(\bPi)$ and $\{\bz\s\}_{s=1,\dots,N}$ provide a stable splitting as defined in the lemma, then 
\begin{align*}
\langle \bx, \bA \bx \rangle = \sum_{s=1}^N \langle \bx, \bA \bPi \bR\s^\top \bz \s  \rangle = \sum_{s=1}^N \langle \tilde\bA\s^\dagger   \bR\s \bPi^\top \bA \bx, \tilde\bA\s\bz \s  \rangle. 
\end{align*}
Indeed $\tilde\bA\s\tilde\bA\s^\dagger   \bR\s \bPi^\top =  \bR\s \bPi^\top$ holds because of \eqref{eq:prop-pseudoinv} and
\[
 \range(  \bR\s \bPi^\top) =  \left(\operatorname{Ker} (\bPi \bR\s^\top)\right)^{\perp^{\ell_2}} \subset    \left(\operatorname{Ker} (\tilde \bA \s)\right)^{\perp^{\ell_2}}  =  \range(\tilde \bA\s) , 
\]
recalling that $\bR\s^\top \operatorname{Ker}(\tilde \bA \s) \subset V\0 = \operatorname{Ker}(\bPi)$. 
Next, the generalized Cauchy-Schwarz inequality for the semi-norm induced by $\tilde \bA \s$, the first property in \eqref{eq:prop-pseudoinv}, the Cauchy-Schwarz inequality in the $\ell_2$-inner product, and the stable splitting assumption are applied in order to get 
\begin{align*}
\langle \bx, \bA \bx \rangle &\leq \sum_{s=1}^N \langle \tilde\bA\s^\dagger   \bR\s \bPi^\top \bA \bx, \tilde\bA\s \tilde\bA\s^\dagger   \bR\s \bPi^\top \bA \bx \rangle^{1/2} \langle \bz\s, \tilde\bA\s\bz \s  \rangle^{1/2} \\
 &\leq \sum_{s=1}^N \langle \tilde\bA\s^\dagger   \bR\s \bPi^\top \bA \bx, \bR\s \bPi^\top \bA \bx \rangle^{1/2} \langle \bz\s, \tilde\bA\s\bz \s  \rangle^{1/2} \\
&\leq \left[\sum_{s=1}^N \langle \tilde\bA\s^\dagger   \bR\s \bPi^\top \bA \bx, \bR\s \bPi^\top \bA \bx \rangle \right]^{1/2} \left[\sum_{s=1}^N  \langle \bz\s, \tilde\bA\s\bz \s  \rangle\right]^{1/2} \\ 
&\leq \langle\bx, \bA\bH\bA \bx \rangle^{1/2} C_0 \langle \bx , \bA \bx \rangle^{1/2}.
\end{align*}
Squaring and cancelling the common factor $\langle \bx , \bA \bx \rangle$ ($\neq 0$ if $\bx \neq \mathbf 0$) yields
\begin{equation}
\label{eq:intermediate-lambdamin}
 \langle\bx, \bA\bH\bA \bx \rangle \geq C_0^{-2} \langle \bx, \bA \bx \rangle, \text{ for any } \bx \in \range(\bPi). 
\end{equation}
Finally, the bound for $\lambda_{\min}$ is proved starting with the definition of an eigenvalue: 
\begin{align*}
\lambda_{\min} \text{ is an eigenvalue of $\bH \bA \bPi$ } &\Leftrightarrow \, \exists\, \bx \in \mathbb R^n; \, \bx \neq \mathbf{0} \text{ such that } \bH \bA \bPi  \bx = \lambda_{\min} \bx.
\end{align*}
Let $\bx$ be such an eigenvector corresponding to eigenvalue $\lambda_{\min}$. By definition, $\lambda_{\min} \neq 0$ so $\bPi \bx \neq \mathbf 0$. Taking the inner product by $\bA\bPi\bx$ gives
\[
\langle \bPi \bx,\bA  \bH \bA \bPi  \bx \rangle = \lambda_{\min} \langle \bA \bPi \bx , \bx \rangle = \lambda_{\min} \langle \bA \bPi \bx , \bPi \bx \rangle \leq \lambda_{\min} C_0^2 \langle \bPi \bx,\bA  \bH \bA \bPi  \bx \rangle, 
\]
where the inequality comes from \eqref{eq:intermediate-lambdamin}. Cancelling the common factor $\langle \bPi \bx,\bA  \bH \bA \bPi  \bx \rangle \neq 0$, leads to the conclusion that $\lambda_{\min} C_0^2 \geq 1$. 
\end{proof}
\begin{theorem}[Spectral bounds for $\bH \bA \bPi$ $\Rightarrow$ Spectral bounds for $\bH_{hyb} \bA$]
\label{th:projhyb}
Let Assumptions~\ref{ass:RsVs-Astilde-Hspd} and \ref{ass:V0} hold. If the eigenvalues of the projected and preconditioned operator satisfy
\[
\lambda(\bH \bA \bPi) \in \{0\} \cup [\lambda_{\min}, \lambda_{\max}],  
\]
then the eigenvalues of the operator preconditioned by $\bH_{hyb}$ from Definition~\ref{def:precs} satisfy
 \[
\lambda(\bH_{hyb} \bA ) \in [\min (1,\lambda_{\min}), \max(1,\lambda_{\max})].
\]
\end{theorem}
\begin{proof}
The connection between the spectra of the projected and hybrid/balanced preconditioned operators is well known (see \textit{e.g.}, \cite{tang2009comparison,klawonn2012deflation}) and easy to verify. Let $\bx \in \mathbb R^n$, it holds that 
\begin{align*}
\langle \bx, \bA\bH_{hyb}\bA \bx \rangle 
& =\langle \bx, \bA\bPi \bH \bPi^\top \bA \bx \rangle + \langle \bx, \bA\bR\0^\top (\bR\0 \bA \bR\0^\top)^{-1} \bR\0 \bA \bx \rangle\\
&=\langle \bPi \bx, \bA \bH \bA \bPi \bx \rangle + \langle (\matid - \bPi) \bx, \bA (\matid - \bPi) \bx \rangle, 
\end{align*}
so, with the result for the projected preconditioned operator:
\[
\left\{
\begin{array}{rc}
&  \langle \bx, \bA\bH_{hyb}\bA \bx\rangle    \geq \lambda_{\text{min}} \langle \bPi \bx, \bA \bPi \bx \rangle + \langle (\matid - \bPi) \bx, \bA (\matid - \bPi) \bx \rangle
\\
\text{and }& \langle \bx, \bA\bH_{hyb}\bA \bx \rangle \leq \lambda_{\text{max}} \langle \bPi \bx, \bA \bPi \bx \rangle + \langle (\matid - \bPi) \bx, \bA (\matid - \bPi) \bx \rangle 
\end{array}
\right.
\]
and the result follows by recalling that $\bPi$ is an $\bA$-orthogonal projection so $ \langle \bx, \bA \bx \rangle =  \langle \bPi \bx, \bA \bPi \bx \rangle + \langle (\matid - \bPi) \bx, \bA (\matid - \bPi) \bx \rangle $. 
\end{proof}

\subsection{Proof of \eqref{eq:th1-lambdamax} for $\bH \bA\bPi$}

\begin{theorem}
\label{th:lambdamaxHApi}
Let Assumptions~\ref{ass:RsVs-Astilde-Hspd} and \ref{ass:V0} hold. Let $\tau >0$. 
If for any $s \in \unN$, it holds that  $\left(\bR\s^\top \mathcal Y_L(\tau, \tilde \bA \s, \bR\s \bA  \bR\s^\top)\right)  \subset V\0 $, then the largest eigenvalue $\lambda_{\max}$ of $\bH\bA\bPi$ satisfies: $\lambda_{\max} \leq {\Ncol}{\tau}^{-1}$.   
\end{theorem}
\begin{proof}
It is assumed that $\tau > 0$, so for each $s=1,\dots,N$, $\operatorname{Ker}(\tilde \bA\s) \subset \mathcal Y_L(\tau, \tilde \bA \s, \bR\s \bA  \bR\s^\top)$ and
\[
\bR\s^\top \operatorname{Ker}(\tilde \bA\s) \subset V\0.
\]
Then, according to the result in Lemma~\ref{lem:upper}, a sufficient condition for the result in the theorem is that, for any $s = 1,\dots,N$, 
\begin{equation}
\label{eq:stablocalsolverwithVsharp}
\bx \s \in \range (\tilde \bA\s^\dagger \bR\s \bPi^\top) \Rightarrow \| \KPs \bR\s^\top \bx \s \|_{\bA}^2 \leq \tau^{-1}{\Ncol} | \bx\s |_{\tilde \bA \s}^2. 
\end{equation}
Recall that $\KPs$ was defined in Lemma~\ref{lem:upper}. Let $s= 1,\dots,N$ be fixed and, for the length of the proof, let $\mathcal Y_L =  \mathcal Y_L(\tau, \tilde \bA \s, \bR\s \bA  \bR\s^\top)$ and $\mathcal Y_H =  \mathcal Y_H(\tau, \tilde \bA \s, \bR\s \bA  \bR\s^\top)$ in order to shorten notations. 

We first characterize the space $ \range (\tilde \bA\s^\dagger \bR\s \bPi^\top)$. The assumption is that $\bR\s^\top \mathcal Y_L  \subset V\0$, so  
 $\bR\s^\top \mathcal Y_L \subset  \operatorname{ker}(\bPi) = \left(\range(\bPi^\top)\right)^{\perp^{\ell_2}}$ 
which implies that
 $\mathcal Y_L \subset  \left(\range(\bR\s \bPi^\top)\right)^{\perp^{\ell_2}}$  
where the $\ell_2$-orthogonality is now in $\mathbb R^{n\s}$ instead of $\mathbb R^n$. Taking the orthogonal again and applying \eqref{eq:conjugacy} from Lemma~\ref{lem:gevp-YLYH}  yields
\[
\range(\bR\s \bPi^\top) \subset (\mathcal Y_L)^{\perp^{\ell_2}} = \bR\s \bA \bR\s^\top \mathcal Y_H =  \tilde \bA\s \mathcal Y_H. 
\] 
It then follows, by definition of $\tilde \bA \s \,^ \dagger$, that
\[
 \range(\tilde \bA\s^\dagger \bR\s \bPi^\top) \subset \tilde \bA\s^\dagger  \tilde \bA\s \mathcal Y_H  \subset \mathcal Y_H + \operatorname{Ker}(\tilde \bA \s).
\]

Now, let $\bx\s \in  \range(\tilde \bA\s^\dagger \bR\s \bPi^\top)$. It has just been proved that there exist $\by\s \in \mathcal Y_H$ and $\bz\s \in \operatorname{Ker}(\tilde \bA \s)$ such that $\bx\s = \by\s + \bz\s$, 
so $\KPs \bR\s^\top \bx\s = \KPs \bR\s^\top \by\s$. Moreover, $\KPs$ being an $\bA$-orthogonal projection, its range is the space
\[
\range (\KPs) = \left(\operatorname{Ker} (\KPs) \right)^{\perp^\bA} = \left(\bR\s^\top \operatorname{Ker}( \tilde \bA\s) \right)^{\perp^\bA} \supset \left(\bR\s^\top \mathcal Y_L\right)^{\perp^\bA}.
\]
The last inclusion follows from $\operatorname{Ker}( \tilde \bA\s) \subset \mathcal Y_L$. Another application of \eqref{eq:conjugacy} from Lemma~\ref{lem:gevp-YLYH} guarantees that ${\mathcal Y_H}^{\perp^{\ell_2}} = \bR\s \bA \bR\s^\top\mathcal Y_L$ so 
 $\bR\s^\top\mathcal Y_H \subset \left(\bR\s^\top \mathcal Y_L\right)^{\perp^\bA} \subset \range (\KPs)$. 
Consequently, $\KPs \bR\s^\top \bx\s = \bR\s^\top \by\s$ and the desired estimate can finally be proved, using the second spectral estimate from \eqref{eq:defYLYH} in Lemma~\ref{lem:upper} to get the inequality,  
\[
\| \KPs \bR\s^\top \bx \s \|_{\bA}^2 = \|\bR\s^\top \by\s \|_{\bA}^2 =\|\by\s \|_{\bR\s\bA\bR\s^\top}^2 \leq \tau^{-1} | \by\s |_{\tilde \bA \s}^2 = \tau^{-1}  | \bx \s |_{\tilde \bA \s}^2.  
\]
\end{proof}

\subsection{Proof of \eqref{eq:th1-lambdamin} for $\bH \bA\bPi$}
\begin{theorem}
\label{th:lambdaminHApi}
Let Assumptions~\ref{ass:RsVs-Astilde-Hspd} and \ref{ass:V0} hold. Let $\tau >0$. 
Under Assumption~\ref{ass:stable-split-matrix} (or \ref{ass:stable-split-matrix-withDs}), if for any $s \in \unN$,
\begin{enumerate}
\item \label{case1} either  $\tilde \bA \s$ is non-singular, and $\left(\bR\s^\top \mathcal Y_L(\tau^{-1}, \bM \s,  \tilde \bA \s )\right) \subset V\0 $,
\item \label{case2} or $\bW\s \in \mathbb R^{n\s \times \operatorname{rank}( \bM\s)}$ is a matrix whose columns form an $\ell_2$-orthonormal basis of $\range(\bM\s)$
and $\left(\bR\s^\top 
\left[ \operatorname{Ker}(\tilde \bA \s)+  \operatorname{Ker}(\bM\s) +  \bW\s \mathcal Y_H(\tau,  \bW\s^\top \tilde \bA \s \bW\s, \bW\s^\top \bM \s  \bW\s)\right]
 \right) \subset V\0 $
\end{enumerate}
then the smallest non-zero eigenvalue $\lambda_{\min}$ of $\bH\bA\bPi$ satisfies: $  (\tau \Nprime)^{-1}  <  \lambda_{\min}. 
$
\end{theorem}
\begin{proof}
The proof consists in checking that the assumptions in Lemma~\ref{lem:stabsplit} are satisfied. The fact that $\sum_{s=1}^N \bR\s^\top \operatorname{Ker}(\tilde \bA \s) \subset V\0$  is clear. It remains to prove that there exists a stable splitting of any $\bx \in \range(\bPi)$ with $C_0^2 = \tau \Nprime$.  

Letting $\bx \in \range(\bPi)$, the idea is to start with the stable splitting from Assumption~\ref{ass:stable-split-matrix} or Assumption~\ref{ass:stable-split-matrix-withDs} which splits $\bx$ as
\[
\bx = \sum_{s=1}^N \bR\s^\top \by\s \text{ satisfying } \sum_{s=1}^N \langle \by\s , \bM\s \by\s \rangle \leq \Nprime  \langle \bx, \bA \bx \rangle.
\]

If local components $\bz\s$ are found such that $\bPi \bR\s^\top \bz\s = \bPi  \bR\s^\top \by\s$ and $\langle \bz\s, \tilde \bA\s \bz\s \rangle 
\leq \tau \langle \by\s, \bM\s \by\s \rangle$, then it holds both that $ \sum_{s=1}^N \bPi \bR\s^\top \bz\s =  \bPi \sum_{s=1}^N \bR\s^\top \by\s = \bPi \bx = \bx$ and  $\sum_{s=1}^N \langle \bz\s , \tilde\bA\s \bz\s \rangle \leq \Nprime \tau  \langle \bx, \bA \bx \rangle$. In other words finding such $\bz\s$ concludes the proof.

Let $s \in \unN$. We start with the case labelled~\ref{case1} and set $\mathcal Y_L\sups = \mathcal Y_L(\tau,  \bW\s^\top \tilde \bA \s \bW\s, \bW\s^\top \bM \s  \bW\s)$ for the length of the proof. 
Let
\[
\bz\s = \bW\s \bP\s \bW\s^\top \by\s 
\]
where $\bP\s$ is the $(\bW\s^\top \bM \s  \bW\s)$-orthogonal projection onto $\mathcal Y_L\sups$. 

We first check that $ \bPi \bR\s^\top \bz\s = \bPi  \bR\s^\top \by\s$, \text{i.e.}, $\bPi \bR\s^\top \left(\matid - \bW\s \bP\s \bW\s^\top \right) \by\s = 0$. 
 We proceed as follows
\[
\range( \matid - \bW\s \bP\s \bW\s^\top ) = \left( \operatorname{Ker} ( \matid - \bW\s \bP\s^\top \bW\s^\top ) \right)^{\perp^{\ell_2}} 
 \subset \left( \bM \s  \bW\s \mathcal Y_L\sups   \right)^{\perp^{\ell_2}}, 
\]
since $\bW\s \bP\s^\top (\bW\s^\top \bM \s  \bW\s) \mathcal Y_L\sups =  \bW\s(\bW\s^\top \bM \s  \bW\s)  {\bP\s} \mathcal Y_L\sups = \bM \s  {\bW\s} \mathcal Y_L\sups $ which implies that $ (\bM \s  \bW\s \mathcal Y_L\sups)  \subset \operatorname{Ker} ( \matid - \bW\s \bP\s^\top \bW\s^\top ) $. It follows that 
\begin{align*}
\range( \matid - \bW\s \bP\s \bW\s^\top ) & = \operatorname{Ker}({\mathcal Y_L\sups}^\top  \bW\s^\top \bM \s)\\
& = \operatorname{Ker}({\bM\s}) +  \bW\s {\mathcal Y_H(\tau,  \bW\s^\top \tilde \bA \s \bW\s, \bW\s^\top \bM \s  \bW\s))}\\& \subset \operatorname{Ker}(\bPi \bR\s^\top),
\end{align*}
where in the second step, one inclusion is easy to check with \eqref{eq:conjugacy} from Lemma~\ref{lem:gevp-YLYH} and the dimensions of both spaces are equal. The stability property follows from 
\begin{align*} 
\langle \bz\s, \tilde \bA\s \bz\s \rangle & = \langle \bP\s \bW\s^\top \by\s, (\bW\s^\top \tilde \bA\s \bW\s) \bP\s \bW\s^\top \by\s\rangle \\
 & < \tau \langle \bP\s \bW\s^\top \by\s, (\bW\s^\top \bM\s \bW\s) \bP\s \bW\s^\top \by\s\rangle \text{ (by \eqref{eq:defYLYH} in Lemma~\ref{lem:gevp-YLYH})} \\
&\leq \tau \langle \bW\s^\top \by\s, (\bW\s^\top \bM\s \bW\s) \bW\s^\top \by\s\rangle \text{ ($\bP\s$ is a $(\bW\s^\top \bM\s \bW\s)$-orthogonal projection)}\\
&= \tau \langle \by\s, \bM\s \by\s\rangle \text{ since } \bM\s \by\s = \bM\s \bW\s \bW\s^\top \by\s + \bM\s \underbrace{(\matid - \bW\s \bW\s^\top) \by\s}_{\in \operatorname{Ker}(\bM\s)}.
\end{align*}

We now address the case labelled~\ref{case2}. 
 Let
\[
\bz\s = {\bP\s}' \by\s, 
\]
where ${\bP\s}'$ is the $\tilde \bA \s$-orthogonal projection onto $\mathcal Y_H(\tau^{-1},  \bM \s , \tilde \bA \s  ) $. 
We first check that $ \bPi \bR\s^\top \bz\s = \bPi  \bR\s^\top \by\s$, \text{i.e.},
$\bPi \bR\s^\top \left(\matid - {\bP\s}' \right) \by\s = \mathbf 0$. This is indeed the case since $\bPi \bR\s^\top \mathcal Y_L({\tau}^{-1},  \bM \s , \tilde \bA \s  ) = \mathbf 0 $ and
\[
\range( \matid - {\bP\s}' )  = \operatorname{Ker} ( {\bP\s}' ) = \left(\mathcal Y_H({\tau}^{-1},  \bM \s , \tilde \bA \s  )\right)^{\perp^{\tilde\bA\s}}    =\mathcal Y_L({\tau}^{-1},  \bM \s , \tilde \bA \s  ),
\] 
by \eqref{eq:conjugacy} in Lemma~\ref{lem:gevp-YLYH}.  

The stability property follows from  
\begin{align*} 
 \langle \bz\s, \tilde \bA\s \bz\s \rangle & = \langle {\bP\s}' \by\s, \tilde \bA\s {\bP\s}' \by\s\rangle \\
 & \leq \tau  \langle {\bP\s}' \by\s, \bM\s {\bP\s}' \by\s\rangle \text{ (by \eqref{eq:defYLYH} in Lemma~\ref{lem:gevp-YLYH})} \\
 & \leq \tau  \left[\langle {\bP\s}' \by\s, \bM\s {\bP\s}' \by\s\rangle + \langle (\matid - {\bP\s}') \by\s, \bM\s (\matid- {\bP\s}') \by\s\rangle\right]\\
&  \leq \tau  \langle \by\s, \bM\s \by\s\rangle \text{ (by \eqref{eq:conjugacy} in Lemma~\ref{lem:gevp-YLYH})}.
\end{align*}
\end{proof}

\subsection{Proof of \eqref{eq:th2-hard}} 

\textcolor{black}{Spectral results for the two-level additive preconditioner $\bH_{ad}$ without the assumption that the local solvers are for $\bR\s\bA\bR\s^\top$ are a novelty. There is an additional assumption in the form of an additional coarse space.} 

\begin{theorem}[Spectral bound for  $\bH_{hyb} \bA$ $\Rightarrow$ Spectral bound for $\bH_{ad} \bA$]
\label{th:add}
Let Assumptions~\ref{ass:RsVs-Astilde-Hspd}, \ref{ass:V0}, and \ref{ass:stable-split-matrix} (or  \ref{ass:stable-split-matrix-withDs}) hold. Moreover, assume that the matrices $\tilde \bA \s$ are non-singular. If for any $s \in \unN$,  $\left(\bR\s^\top \mathcal Y_L(\upsilon, \tilde \bA \s, \bR\s \bA  \bR\s^\top)\right)  \subset V\0 $ and $0 < \lambda_{\min,hyb} \leq \lambda(\bH_{hyb}\bA)$ is a lower bound for the eigenvalues of the hybrid preconditioned operator then  
\[
 \left[\max\left(2, 1 + 2 {\Ncol}{\upsilon}^{-1}\right)   \lambda_{\min,hyb}^{-1}  \right]^{-1}  \leq \lambda(\bH_{ad} \bA). 
\] 
\end{theorem}
\begin{proof}
The proof comes down to applying Lemma~\ref{lem:stabsplit} (stable splitting). Indeed, the two-level Additive preconditioner fits the abstract framework by considering that there are $N+1$ subspaces ($ \range(\bR\s^\top)$ for $s=0, \dots, N$) that play the same role. In other words, the coarse space $V\0$ is viewed just like any of the other subspaces with the local solver $\tilde\bA \0 = {\bR\0 \bA \bR\0}^\top $ and the interpolation operator $\bR\0^\top$. There is no coarse space that is treated by projection so the projection operator in Lemma~\ref{lem:stabsplit} equals identity. Let $\bx \in \mathbb R^n$, it suffices to prove that there exist $\bz\s \in \mathbb R^{n\s}$ for any $s=0,\dots,N$ such that
\begin{equation}
\label{eq:whatwewant}
\bx =\sum_{s=0}^N \bR\s^\top \bz\s  \text{ and } \sum_{s=1}^N \| \bz\s \|_{\tilde \bA \s}^2 + \|\bz\0\|_{\tilde \bA \0}^2 \leq 
C_0^2 \|\bx \|_\bA^2; \quad C_0^2 =\max\left(2, 1 + 2 {\Ncol}{\upsilon}^{-1}\right)   \lambda_{\min,hyb}^{-1}. 
\end{equation}
Inspired by \cite{ToselliWidlund_book2005}[Lemma 2.5], the proof starts with $\langle \bx, \bH_{hyb}^{-1} \bx \rangle \leq \lambda_{\min,hyb}^{-1} \langle \bx, \bA \bx \rangle$ and follows with 
\begin{align*}
\langle \bx, \bH_{hyb}^{-1} \bx \rangle &= \langle \bH_{hyb}^{-1} \bx, \bH_{hyb} \bH_{hyb}^{-1} \bx \rangle \\
&= \sum_{s=1}^N \langle \bH_{hyb}^{-1} \bx, \bPi \bR\s^\top \tilde\bA\s^{-1} \bR\s \bPi^\top \bH_{hyb}^{-1} \bx \rangle + \langle \bH_{hyb}^{-1} \bx, \bR\0^\top \tilde\bA\0^{-1} \bR\0 \bH_{hyb}^{-1} \bx \rangle\\
&= \sum_{s=1}^N \langle \bR\s \bPi^\top \bH_{hyb}^{-1} \bx, \tilde\bA\s^{-1} \tilde \bA\s \tilde\bA\s^{-1} \bR\s \bPi^\top \bH_{hyb}^{-1} \bx \rangle + \langle \bH_{hyb}^{-1} \bx, \bR\0^\top \tilde\bA\0^{-1} {\tilde\bA\0} \tilde\bA\0^{-1} \bR\0 \bH_{hyb}^{-1} \bx \rangle\\ 
& = \sum_{s=1}^N \|\bz\s\|_{\tilde \bA \s}^2 + \| {\bz\0}'\|_{\tilde\bA \0}^2,
\end{align*}
with $\bz\s := \tilde\bA\s^{-1} \bR\s \bPi^\top \bH_{hyb}^{-1} \bx$ for $s=1, \dots,N$, and ${\bz\0}' := \tilde\bA\0^{-1} \bR\0 \bH_{hyb}^{-1} \bx  $. This first splitting of $\bx$ satisfies  $ \bx = \sum_{s=1}^N \bPi \bR\s^\top \bz\s + \bR\0^\top {\bz\0}'$ and 
\begin{equation}
\label{eq:almost-splitting}
\sum_{s=1}^N \|\bz\s\|_{\tilde \bA \s}^2 + \| {\bz\0}'\|_{\tilde\bA \0}^2 \leq \lambda_{\min,hyb}^{-1} \langle \bx, \bA \bx \rangle.
\end{equation}
The only problem is the presence of $\bPi$. Instead, the splitting of $\bx$ is rewritten to suit the fully additive setting, \text{i.e.}, so that it satisfies \eqref{eq:whatwewant}:
\[
\bx = \sum_{s=1}^N \bR\s^\top \bz\s - (\matid - \bPi) \sum_{s=1}^N \bR\s^\top \bz\s +\bR\0^\top  {\bz\0}' =  \sum_{s=1}^N \bR\s^\top \bz\s + \bR\0^\top \underbrace{\left[- (\bR\0 \bA \bR\0^\top)^{-1} \bR\0 \bA  \sum_{s=1}^N \bR\s^\top \bz\s + {\bz\0}'  \right]}_{{:=\bz\0 \, \in V\0}}. 
\]
It remains to prove that the splitting satisfies~\eqref{eq:whatwewant}. 
To this end we calculate
\begin{align*}
\|\bz\0\|_{\tilde\bA\0}^2 &\leq 2 \|{\bz\0}'\|_{\tilde \bA \0}^2 +  2 \| (-\bR\0 \bA \bR\0^\top)^{-1} \bR\0 \bA  \sum_{s=1}^N \bR\s^\top \bz\s    \|_{\tilde \bA\0}^2 \\
&=  2 \|{\bz\0}'\|_{\tilde \bA\0}^2 +  2 \langle  \sum_{s=1}^N \bR\s^\top \bz\s , \bA \bR\0^\top(\bR\0 \bA \bR\0^\top)^{-1} \bR\0 \bA  \sum_{s=1}^N \bR\s^\top \bz\s \rangle \\ 
&=  2 \|{\bz\0}'\|_{\tilde \bA\0}^2 +  2 \| (\matid - \bPi)  \sum_{s=1}^N \bR\s^\top \bz\s \|_\bA^2 \\ 
&=  2 \|{\bz\0}'\|_{\tilde \bA\0}^2 +  2 \| (\matid - \bPi)  \bH \bPi^\top \bH_{hyb}^{-1} \bx \|_\bA^2 \\
&\leq  2 \|{\bz\0}'\|_{\tilde \bA\0}^2 +  2 \| \bH \bPi^\top \bH_{hyb}^{-1} \bx \|_\bA^2 \\
& \leq  2 \|{\bz\0}'\|_{\tilde \bA\0}^2 +  2 \Ncol \sum_{s=1}^N \| \bR\s^\top \tilde\bA\s^{-1} \bR\s \bPi^\top \bH_{hyb}^{-1} \bx \|_\bA^2 \text{ (Cauchy-Schwarz with Definition~\ref{def:color} of $\mathcal  N$)}  \\
& \leq  2 \|{\bz\0}'\|_{\tilde \bA\0}^2 +  2 \Ncol \sum_{s=1}^N \| \tilde\bA\s^{-1}  \bR\s \bPi^\top \bH_{hyb}^{-1} \bx \|_{\bR\s \bA\bR\s^\top}^2   \\
& \leq   2 \|{\bz\0}'\|_{\tilde \bA\0}^2 +  2 {\Ncol}\upsilon^{-1} \sum_{s=1}^N \| \tilde\bA\s^{-1}  \bR\s \bPi^\top \bH_{hyb}^{-1} \bx \|_{\tilde\bA\s}^2 \text{(\eqref{eq:stablocalsolverwithVsharp} in the proof of Theorem~\ref{th:lambdamaxHApi} with $\KPs = \matid$)}  \\
& =  2 \|{\bz\0}'\|_{\tilde \bA \0}^2 +  2 {\Ncol}{\upsilon}^{-1} \sum_{s=1}^N \| \bz\s\|_{\tilde\bA\s}^2 . 
 \end{align*}

Finally, by putting this together with \eqref{eq:almost-splitting}, it follows that 
\begin{align*}
\sum_{s=1}^N \|\bz\s\|_{\tilde \bA \s}^2 + \| {\bz\0}\|_{\tilde \bA\0}^2 &\leq  \sum_{s=1}^N \|\bz\s\|_{\tilde \bA \s}^2 +  2 \|{\bz\0}'\|_{\tilde \bA\0}^2 +  2 {\Ncol}{\upsilon}^{-1} \sum_{s=1}^N \| \bz\s\|_{\tilde\bA\s}^2  \\ 
&= \left( 1 + 2 {\Ncol}\upsilon^{-1}\right)  \sum_{s=1}^N \|\bz\s\|_{\tilde \bA \s}^2 +  2 \|{\bz\0}'\|_{\tilde \bA\0}^2   \\ 
&\leq \max\left(2, 1 + 2 {\Ncol}{\upsilon}^{-1}\right)  \lambda_{\min,hyb} ^{-1} \langle \bx, \bA \bx \rangle. 
\end{align*}
\end{proof}

\section{Example: 2d linear elasticity with Additive Schwarz,\\ Neumann-Neumann and Inexact Schwarz}
\label{sec:examples}

In this Section, the abstract framework is made concrete. Its setup, analysis and numerical performance are described for a two-dimensional linear elasticity problem.

\subsection{Geometry and PDE}
\label{sub:elast}
Let $\Omega = [0 , 2] \times [0,1]  \subset \mathbb R^2$  be the computational domain. Let $\partial \Omega_D$ be the left hand side boundary of $\Omega$ and let $\mathcal V =  \{\bv \in H^1(\Omega)^2; \bv= \mathbf 0 \text{ on } \partial\Omega_D\}$. The linear elasticity equations posed in $\Omega$  with mixed boundary conditions are considered. A solution $\bu  \in \mathcal V$ is sought such that
\begin{equation}
\label{eq:elasticity}
\int_{\Omega} 2 \mu \varepsilon(\bu) : \varepsilon(\bv) \, dx +  \int_{\Omega} L \operatorname{div}(\bu) \operatorname{div}(\bv) \, dx  = \int_{\Omega} \mathbf{g} \cdot \bv \, dx, \text{ for all } \bv \in \mathcal V,
\end{equation}
where, for $i,j=1,2$, $\varepsilon_{ij}(\bu) = \frac{1}{2}\left(\frac{\partial {u}_i}{\partial x_j}+\frac{\partial {u}_j}{\partial x_i}\right)$, $\mathbf g = (0,1)^\top$ and the Lam\'e coefficients are functions of Young's modulus $E$ and Poisson's ratio $\nu$ : $  
\mu = \frac{E}{2(1+\nu)},\, L = \frac{E\nu}{(1+\nu)(1-2\nu)}$. 

\subsection{Discretization}

The computational domain is discretized by a uniform mesh with element size $h = 1/42$ and the boundary value problem is solved numerically with standard piecewise linear ($\mathbb P_1$) Lagrange finite elements. Let $\mathcal V_h$ be the space of $\mathbb P_1$ finite elements that satisfy the Dirichlet boundary condition. Let $\{\bphi_k\}_{k=1}^n$ be a basis of $\mathcal V_h$. The linear system that is to be solved is
\[
\text{Find } \bx \in \mathbb R^n \text{ such that } \bA \bx = \bb,
\] 
with $A_{ij} = \int_{\Omega} \left[ 2 \mu \varepsilon(\bphi_i) : \varepsilon(\bphi_j) +  L \operatorname{div}(\bphi_i) \operatorname{div}(\bphi_j) \right] \, dx $ and $\bb_i = \int_{\Omega} \mathbf{g} \cdot \bphi_i \, dx$.
The dimension of the global problem is $n = 43 \times 84 \times 2 = 7224$ where it has been taken into account that there are two degrees of freedom at each grid point (the $\bx$ and $\by$ displacements) and that there are no degrees of freedom where a Dirichlet boundary condition has been prescribed.

\subsection{Domain Decomposition}
\label{sub:DD}

The computational domain $\Omega$ is partitioned into $N = 8$ non-overlapping subdomains with Metis \cite{METIS} (see Figure~\ref{fig:geom}--left). The geometric subdomains are assumed to be mesh-conforming and they are denoted $\Omega\s$ for $s\in \llbracket 1, N \rrbracket$. Let $V\s = \{S^1\s, \dots, S^{n\s}\s\}$ be the set of mesh nodes that are in each $\Omega\s$ (for $s \in \llbracket 1, N \rrbracket$). These are also the local degrees of freedom. The restriction matrices $\bR\s \in \mathbb R^{n\s \times n}$ are defined by 
\[
(\bR\s)_{ij} = 1 \text{ if } j = S^i\s \text{ and } (\bR\s)_{ij} = 0 \text{ otherwise}. 
\]
Each $\bR\s$ has exactly one $1$ per row. These are in agreement with Assumption~\ref{ass:RsVs-Astilde-Hspd}. By construction, the degrees of freedom that are on the interfaces between subdomains are duplicated (and only these ones). 

We may also assemble the matrices that correspond to the discretization of the problem $\eqref{eq:elasticity}$ restricted to each subdomain $s \in \llbracket 1, N \rrbracket$: $\bANeu \in \mathbb R^{n\s \times n\s}$  such that  
\begin{equation}
\label{eq:def-Aneus}
(  \bANeu )_{ij} := \int_{\Omega\s} \left[ 2 \mu \varepsilon(\bphi_{S^i\s}) : \varepsilon(\bphi_{S^j\s})  +   L \operatorname{div}(\bphi_{S^i\s}) \operatorname{div}(\bphi_{S^j\s}) \, dx \right] \text{ for all } i,j.
\end{equation}

These are frequently referred to as the local Neumann matrices as they arise from assembling the original problem over the subdomain $\Omega\s$ with natural boundary conditions. They can't be computed from the global matrix $\bA$. Since we consider non-overlapping subdomains they satisfy the very useful property that
\[
\bA = \sum_{s=1}^N \bR\s^\top \bANeu \bR\s.
\]

It is chosen to fulfill Assumption~\ref{ass:stable-split-matrix} through the use of a partition of unity as proposed in Assumption~\ref{ass:stable-split-matrix-withDs}. Let the partition of unity be 
\begin{itemize}
\item either the multiplicity scaling, also called the $\mu$-scaling
\begin{equation}
\label{eq:muscaling}
\bD\s :=  \bR \s \left(\sum_{t=1}^N \bR \st^\top  \bR \st \right)^{-1} \bR\s^\top, \, \forall s \in \unN, 
\end{equation}
\item or the $k$-scaling
\begin{equation}
\label{eq:kscaling}
\bD\s \in \mathbb R^{ n\s \times  n\s} \text{ diagonal with entries } (\bD\s)_{ii} := \frac{(\bANeu)_{ii}}{(\bR\s \bA \bR\s^\top)_{ii}}, \, \forall i \in \llbracket 1, n\s \rrbracket, \, \forall s \in \unN. 
\end{equation}
\end{itemize}
Assumption~\ref{ass:stable-split-matrix} is fulfilled with $\Nprime = 1$ by setting 
\begin{equation}
\label{eq:Ms}
\bM\s := \bD\s^{-1} \bANeu \bD\s^{-1}; \text{ and } \by\s  := \bD\s \bx, \, \forall \bx \in \mathbb R^n. 
\end{equation}

These matrices are typically singular and their kernel is the set of rigid body modes on the subdomain weighted by the partition of unity. 

\subsection{Choice of parameters in the PDE}
\begin{figure}
\begin{center}
\begin{tabular}{ccc}
\includegraphics[width=0.23\textwidth]{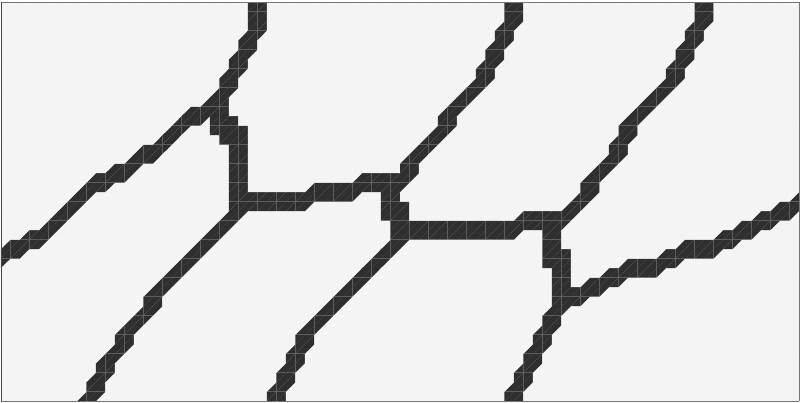}&
\includegraphics[width=0.23\textwidth]{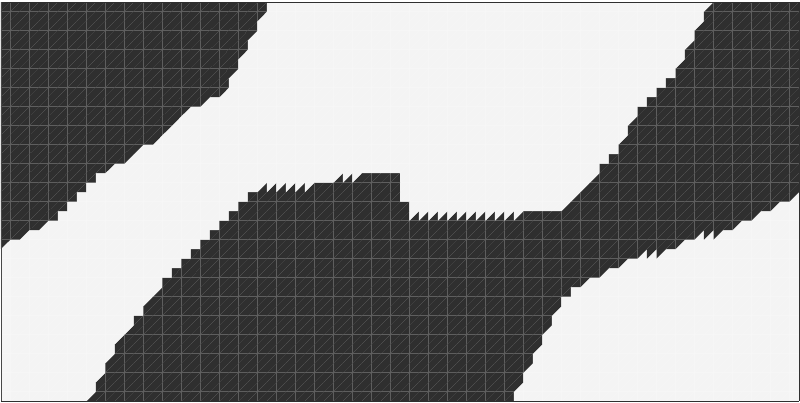}&
\includegraphics[width=0.23\textwidth]{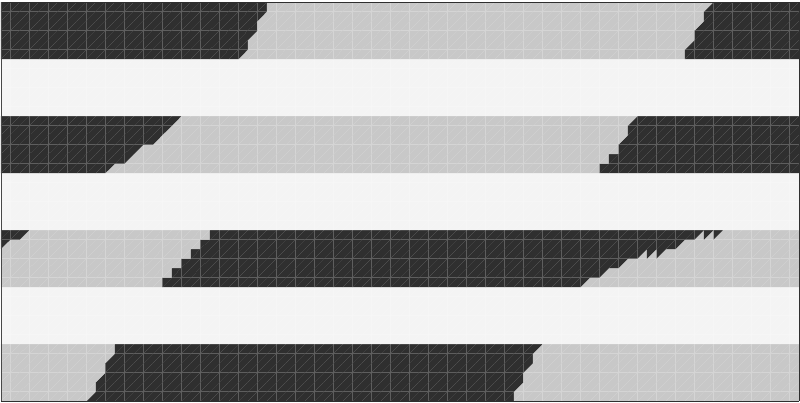}\\
$N=8$ subdomains & $E = 10^5$ (dark) or $10^8$ (light) & $E \mathrel{{+}{=}} 10^9$ in white stripes  \\
& `no layers' & `with layers'  
\end{tabular}
\end{center}
\caption{Partition into subdomains, distribution of E without and with harder layers.}
\label{fig:geom}
\end{figure}
We set Poisson's ratio to $\nu = 0.4$ in all of the domain for all test cases. Two distributions of Young's modulus are considered. In the first $E$ is constant per subdomain: $E = 10^5$ if $s$ is odd and $E = 10^8$ if $s$ is even. The second distribution of $E$ is obtained by adding some rigid layers to the first one: $E$ is augmented by $10^{9}$ if $y \in [1/7, 2/7] \cup [3/7, 4/7] \cup [5/7, 6/7]$. These test cases are referred to as `no layers' and `with layers' and the coefficient distributions are plotted in Figure~\ref{fig:geom}. It is well known (see, \textit{e.g.}, \cite{pechstein2011analysis}) that the solution of \eqref{eq:elasticity} in a heterogeneous medium is challenging.

\subsection{Domain Decomposition Preconditioners}
\textcolor{black}{The definitions and spectral results for the three considered domain decomposition methods are summarized in Table~\ref{tab:abstract-app} and presented in detail next.}
\begin{table}
\caption{\textcolor{black}{Three domain decomposition preconditioners, their GenEO coarse spaces and spectral bounds (valid for non-overlapping subdomains) from Theorems~\ref{th:Ad}, \ref{th:NN}, and~\ref{th:IS}. Recall that for any matrix $\bB$, the notation $\lambda(\bB)$ refers to any of its eigenvalues. For the projected operator $\bH \bA \bPi$, the lower bounds are to be understood for any non-zero eigenvalue.}}
\label{tab:abstract-app}
\centering
\scalebox{0.8}{
  \SetTblrInner{rowsep=5pt}
\begin{tblr}{hline{1-2,Z} = {2pt}, 
  hline{3-Y} = {1pt},
  vline{1-Z} = {1-2}{solid},
  vline{1-Z} = {1pt},
  colspec={cccc}}
DD Preconditioner & Local problem $\tilde \bA \s$ & \makecell{Local contributions to\\ GenEO coarse space} & Spectral bounds for two-level prec. \\
\makecell{\textbf{AS} \\
Additive Schwarz} & $ \bR\s \bA \bR\s^\top$   &   \makecell{$\mathcal Y_L(\tau^{-1}, \bD\s^{-1} \bANeu \bD\s^{-1} ,  \bR\s \bA \bR \s^{\top})$\\\smallskip \textit{i.e.} low-frequency vectors of \\ $ \displaystyle \bD\s^{-1} \bANeu \bD\s^{-1} \bx\s = \lambda\s\,  \bR\s \bA \bR\s^\top \bx\s$   } &  \makecell{
$\begin{aligned}
&  \lambda(\bH \bA ) &\leq \Ncol \\
1/\tau   \leq &\lambda(\bH \bA \bPi) &\leq \Ncol \\ 
 1/\tau \leq & \lambda(\bH_{hyb} \bA)  &\leq \Ncol \\ 
((1 + 2 \Ncol) \tau))^{-1} \leq & \lambda(\bH_{ad} \bA )  &\leq \Ncol + 1 
\end{aligned}$
}\\
\makecell{\textbf{NN} \\
Neumann-Neumann} & $\bD\s^{-1} \bANeu \bD\s^{-1}$  &  Same as above &  \makecell{
$\begin{aligned}
1 \leq &\lambda(\bH \bA \bPi) &\leq \Ncol \tau\\ 
 1 \leq & \lambda(\bH_{hyb} \bA ) &\leq \Ncol \tau 
\end{aligned}$
}\\
\makecell{\textbf{IS} \\
Inexact Schwarz} &
\makecell{
$ \bL\s \bL\s^\top \approx \bR\s \bA \bR\s^\top $ \\ (no-fill \\ incomplete \\ Cholesky fact.)}
  &
\makecell{ $ \mathcal Y_L(\upsilon, \bL\s \bL\s^\top , \bR\s \bA  \bR\s^\top) +$ \\ $\mathcal Y_L(\tau^{-1}, \bD\s^{-1} \bANeu \bD\s^{-1},    \bL\s \bL\s^\top)$ \\
\smallskip \textit{i.e.} low-frequency vectors of \\ $ \displaystyle \bL\s \bL\s^\top \bx\s = \lambda\s\, \bR\s \bA  \bR\s^\top \bx\s$ \\ and of \\  $ \displaystyle  \bD\s^{-1} \bANeu \bD\s^{-1}  \bx\s = \lambda\s\,  \bL\s \bL\s^\top  \bx\s $ 
}
   &  \makecell{
$\begin{aligned}
1/\tau \leq &\lambda(\bH \bA \bPi) &\leq \Ncol/ \upsilon\\
1/\tau \leq & \lambda(\bH_{hyb} \bA ) &\leq \Ncol/ \upsilon\\ 
((1 + 2 \Ncol/\upsilon)\tau)^{-1} \leq  & \lambda(\bH_{ad} \bA )  &\leq  \Ncol/C_\sharp + 1 \\ 
\end{aligned}$
}\\
\end{tblr}
}

\bigskip
{\small
Summary of Notation:
\begin{itemize}
\item $\bA$: problem matrix in the linear system $\bA \bx = \bb$,
\item $ \bR\s \bA  \bR\s^\top$: restriction of the problem to subdomain number $s$ where $\bR\s$ is the restriction to the subdomain (see Assumption~\ref{ass:RsVs-Astilde-Hspd}), 
\item $ \bANeu$: local Neumann matrix (defined in~\eqref{eq:def-Aneus}),
\item $ \bD\s$: partition of unity (either defined by~\eqref{eq:muscaling} or by~\eqref{eq:kscaling}),
\item $\bL\s$: factor in the no-fill incomplete Cholesky factorization  $\bL\s \bL\s^\top \approx \bR\s \bA  \bR\s^\top$ \cite{chan1997approximate} (triangular matrix),
\item $\bPi$: coarse projector (Definition~\ref{def:precs}), 
\item $\tilde \bA \s$: local solvers in the definitions of the one-level preconditioner $\bH = \sum_{s=1}^N \bR\s^\top \tilde \bA\s^\dagger \bR\s$ (Definition~\ref{def:precs}),
\item $\bH_{hyb} = \bPi \bH \bPi^\top + \bR\0^\top (\bR\0 \bA \bR\0^\top)^{-1} \bR\0 $: two-level hybrid preconditioner (Definition~\ref{def:precs}),
\item $\bH_{ad} =  \bH + \bR\0^\top (\bR\0 \bA \bR\0^\top)^{-1} \bR\0$: two-level additive preconditioner (Definition~\ref{def:precs}),
\item $\Ncol$: coloring constant (Definition~\ref{def:color}),
\item $\tau > 1$, $\upsilon \in ]0; \Ncol[ $: thresholds chosen by the user,
\item $\mathcal Y_L(\tau, \bM_\bA, \bM_\bB) = \operatorname{span}\{\by; \, \bM_\bA \by = \lambda \bM_\bB \by \text{ with } \lambda < \tau \} $: span of low-frequency eigenvectors  (Definition~\ref{def:YLYH}),
\item $C_\sharp$ : constant defined in Definition~\ref{def:Csharp}. 
\item $\Nprime$ and $\bM\s$ from Assumption~\ref{ass:stable-split-matrix} (or \ref{ass:stable-split-matrix-withDs})),
\end{itemize}
}
\end{table}

\subsubsection{Additive Schwarz preconditioner}

\begin{theorem}[Additive Schwarz preconditioner]
\label{th:Ad}
Let the Additive Schwarz preconditioner be defined by the choice ${\tilde \bA \s} = \bR\s \bA \bR\s^\top$ for every $s \in \llbracket 1, N \rrbracket$, which leads to the one-level preconditioner 
\begin{equation}
\label{def:HAS}
\bH := \sum_{s=1}^N \bR\s^\top(\bR\s \bA \bR\s^\top)^{-1}\bR\s. 
\end{equation}
 
With $\bM\s$ from \eqref{eq:Ms}, and given any threshold $\tau > 1$, let the coarse space be defined 
\begin{itemize}
\item either as  
$V\0 := \sum_{s=1}^N \bR\s^\top \mathcal Y_L(\tau^{-1},  \bM \s ,  \bR\s \bA \bR\s^\top)$,
\item or as $V\0 := \sum_{s=1}^N \bR\s^\top \operatorname{Ker}(\bM\s) + \sum_{s=1}^N \bR\s^\top \bW\s \mathcal Y_H(\tau,  \bW\s^\top \bR\s \bA \bR\s^\top \bW\s, \bW\s^\top \bM \s  \bW\s)$
where the columns of $\bW\s$ form an $\ell_2$-orthonormal basis of $\range(\bM\s)$.
\end{itemize}
 
Under Assumption~\ref{ass:V0}, the two-level operators that result from applying Definition~\ref{def:precs} satisfy 
\begin{eqnarray}
&  \lambda(\bH \bA ) & \leq \Ncol \\
1/\tau   \leq &\lambda(\bH \bA \bPi) &\leq \Ncol  \text{ if } \lambda(\bH \bA \bPi) \neq 0\\ \ 
 1/\tau \leq & \lambda(\bH_{hyb} \bA)  &\leq \Ncol \\ 
1/((1 + 2 \Ncol) \tau)) \leq & \lambda(\bH_{ad} \bA )  &\leq \Ncol + 1 ,
\end{eqnarray}
where $\Ncol$ is the coloring constant from Definition~\ref{def:color}.
\end{theorem}
\begin{proof}
Apply Theorems~\ref{th:main} and~\ref{th:mainadd} (with $C_\sharp = 1$, $\Nprime =1$, $\Ncol \geq 1$). 
 The local matrices $\tilde \bA \s$ in Theorem~\ref{th:Ad} are spd as a result of $\bA$ being spd and all $\bR\s^\top$ being full rank. 
\end{proof}

\begin{remark}[Computation of $V\0$ in Theorem~\ref{th:Ad}] 
\label{rem:computVO}
The coarse space $V\0$ is formed by contributions coming from each subdomain $s \in \llbracket 1, N \rrbracket$. They can be computed as follows. First, a Cholesky factorization with pivoting of the matrix $\bM\s$ is performed. This gives both a factorization of $\bM\s$ and an orthonormal basis $\bZ\s$ for the kernel of $\bM\s$. Then, an eigenvalue problem is solved partially to compute the highest or lowest frequency vectors. In both cases, the factorization of $\bM\s$ is crucial since an application of $\bM\s^{\dagger}$ is necessary at each iteration. Although the first coarse space seems simpler than the second because its presentation is more compact, it does not vary much in actual computation work.  The eigensolver SLEPc \cite{Hernandez:2005:SSF} provides an option for a deflation space. Setting it to $\operatorname{Ker}(\bM\s)$ allows to solve the eigenvalue problem for the second coarse space (with $\bW\s$).
\end{remark}

\begin{remark}[Choice of $\tau$]
\label{rem:tauflatl1}
As $\tau$ decreases, the condition number of the preconditioned operators decreases. But the number of vectors in the coarse space becomes larger. It is not advised to choose $\tau < 1$ as this would lead to a very large coarse space. Indeed, by definition the matrices $\bR\s \bA \bR\s^\top$ and $ \bM\s $ differ only at the interfaces between subdomains so eigenvalue $1$ in the generalized eigenvalue problems is associated with a very large eigenspace. Indeed, it contains all the vectors that are $0$ on the boundary of $\Omega\s$.
\end{remark}

\begin{remark}\label{rem:discharmev} The coarse vectors are $\bR\s \bA \bR\s^\top$-discrete harmonic inside the subdomains. This remark follows from the previous one: all eigenvectors that correspond to an eigenvalue other than $1$ are $\bR\s \bA \bR\s^\top$-orthogonal to all vectors that are supported in the interior of a subdomain. 
\end{remark}

\subsubsection{Neumann-Neumann preconditioner}
\begin{theorem}[Neumann-Neumann preconditioner]
\label{th:NN}
Let the Neumann-Neumann preconditioner be defined by the choice ${\tilde \bA \s} = \bM\s$ with $\bM\s$ from \eqref{eq:Ms} (for every $s \in \llbracket 1, N \rrbracket$), which leads to the one-level preconditioner 
\begin{equation}
\label{def:HNN}
\bH := \sum_{s=1}^N \bR\s^\top\bD\s \bANeu^{\dagger} {\bD \s }\bR\s. 
\end{equation}
 
Given any threshold $1 > \tau > 0$, let the coarse space be defined by 
\[
V\0 := \sum_{s=1}^N \bR\s^\top \mathcal Y_L(\tau,  \bM \s, \bR\s \bA  \bR\s^\top). 
\]
 
Under Assumption~\ref{ass:V0}, the two-level operators that result from applying Definition~\ref{def:precs} satisfy 
\begin{eqnarray} 
1 \leq &\lambda(\bH \bA \bPi) &\leq \Ncol/ \tau \text{ if } \lambda(\bH\bA \bPi) \neq 0\\  
 1 \leq & \lambda(\bH_{hyb} \bA ) &\leq \Ncol/ \tau. 
\end{eqnarray}
where $\Ncol$ is the coloring constant from Definition~\ref{def:color}.
\end{theorem}
Note that, $V\0$ is defined only for $\tau > 0$ so $\sum_{s=1}^N \bR\s^\top \operatorname{Ker} (\bM \s) \subset V\0$.  
\begin{proof}
For the upper bounds, apply Theorem~\ref{th:main}. For the lower bound of $\bH \bA \bPi$, apply Lemma~\ref{lem:stabsplit} with $C_0^2 = 1$ (thanks to $\tilde \bA\s = \bM\s$ and Assumption~\ref{ass:stable-split-matrix} with $\Nprime =1$). The lower bound for $\bH_{hyb} \bA$ follows from Theorem~\ref{th:projhyb}. 
\end{proof}

The choice of partition of unity enters into the one-level preconditioner as well as the coarse space. The local matrices $\tilde \bA \s$ in Neumann-Neumann are singular unless there is a Dirichlet boundary condition for the subdomain numbered $s$.  For the two-dimensional linear elasticity problem, the kernel of $ {\bANeu}$ is the set of rigid body modes. 

A remarkable feature is that the coarse space for Neumann-Neumann is the same as one of the coarse spaces for Additive Schwarz in~Theorem~\ref{th:Ad}: there is a set of coarse vectors that \textit{fixes} both the Neumann-Neumann preconditioners and the Additive Schwarz preconditioners. This was already pointed out in \cite{agullo2019robust}.

The additive version of the Neumann-Neumann preconditioner is not considered because no results can be proved 
(and numerical performance is poor). There is no interesting result for the spectrum of the operator without a coarse space either. The closest thing to that is that $\lambda(\bH \bA \bPi) \geq 1$ with a coarse space consisting only of the kernels of the local solvers.

We refer the reader to the remarks in the previous paragraph which also apply here. In particular, the computation of $V\0$ is similar (or exactly the same), $\tau$ should be chosen to be less than $1$, and the eigenvectors that enter into the coarse space are $\bR\s \bA \bR\s^\top$-harmonic in the interior of the subdomain.

\subsubsection{Inexact Schwarz preconditioner} 
\textcolor{black}{Inexact Schwarz methods are an important family of domain decomposition preconditioners and this is, to the best of the author's knowledge, the first introduction of GenEO coarse spaces for Inexact Schwarz with the incomplete Cholesky factorization.} 

\begin{theorem}[Inexact Schwarz preconditioner]
\label{th:IS}
Let the Inexact Schwarz preconditioner be defined by the choice (for every $s \in \llbracket 1, N \rrbracket$)  
$
\tilde \bA \s = \bL\s \bL\s^\top,
$
where the triangular matrix $\bL\s$ is the factor in the no-fill incomplete \textcolor{black}{Cholesky factorization of $\bR\s\bA{\bR\s}^{\top}$ \cite{chan1997approximate}.} 
This leads to the one-level preconditioner 
\begin{equation}
\label{def:HIS}
\bH := \sum_{s=1}^N \bR\s^\top (\bL\s\bL\s^\top)^{-1}  \bR\s. 
\end{equation}
 
With $\bM\s$ from \eqref{eq:Ms}, and given any two thresholds $\tau > 1$ and $\upsilon \in ]0, \Ncol[$, let the coarse space be defined 
\begin{itemize}
\item either as  
$V\0 := \sum_{s=1}^N  \bR\s^\top \left[ \mathcal Y_L(\upsilon, \bL\s \bL\s^\top , \bR\s \bA  \bR\s^\top) + \mathcal Y_L(\tau^{-1}, \bM \s ,   \bL\s \bL\s^\top) \right]$, 
\item or as $V\0 := \sum_{s=1}^N  \bR\s^\top \left[\operatorname{Ker} (\bM \s) +  \mathcal Y_L(\upsilon, \bL\s \bL\s^\top , \bR\s \bA  \bR\s^\top) + \bW\s \mathcal Y_H(\tau,  \bW\s^\top  \bL\s \bL\s^\top \bW\s, \bW\s^\top \bM \s  \bW\s) \right]$, 
where the columns of $\bW\s$ form an $\ell_2$-orthonormal basis of $\range(\bM\s)$.
\end{itemize}
 
Under Assumption~\ref{ass:V0}, the two-level operators that result from applying Definition~\ref{def:precs} satisfy 
\begin{eqnarray}
1/\tau \leq &\lambda(\bH \bA \bPi) &\leq \Ncol/ \upsilon \text{ if } \lambda(\bH \bA \bPi) \neq 0\\ 
1/\tau \leq & \lambda(\bH_{hyb} \bA ) &\leq \Ncol/ \upsilon \\
\color{black}{((1 + 2 \Ncol/\upsilon)\tau )^{-1} \leq}  & \textcolor{black}{\lambda(\bH_{ad} \bA )}  & \textcolor{black}{\leq  {\Ncol}/{C_\sharp} + 1},
\end{eqnarray}
where $\Ncol$ is the coloring constant from Definition~\ref{def:color} \textcolor{black}{and $C_\sharp$ is as in Definition~\ref{def:Csharp}}.
\end{theorem}
\begin{proof}
Apply Theorems~\ref{th:main} and~\ref{th:mainadd} with $\Nprime = 1$.
\end{proof}

\textcolor{black}{The upper bound for $\lambda(\bH_{ad} \bA )$ involves $C_\sharp$ from Definition~\ref{def:Csharp}, which can be set to be the smallest eigenvalue of $\bL\s \bL\s^\top \bx \s =  \lambda\s \bR\s \bA \bR\s^\top \bx\s $. Although $C_\sharp$ is not known and can't be controlled by the coarse space, it is computed at the same time as $ \mathcal Y_L(\upsilon, \bL\s \bL\s^\top , \bR\s \bA  \bR\s^\top)$. In all of our computations we found $C_\sharp \in [0.45, 0.42]$. Consequently the upper bound for $\lambda(\bH_{ad} \bA )$ is approximately $ 2 \Ncol +1$. The additive version of the two-level inexact Schwarz preconditioner is included in our study.} 

\textcolor{black}{Deflation methods for incomplete Cholesky factorizations have already been considered (without domain decomposition) by \cite{zbMATH01743066}. The authors also observed that it is only the lower part of the spectrum of $\bL\s^{-1}(\bR\s \bA \bR\s^\top) {\bL\s^\top}^{-1} $ that is problematic.}

\begin{remark}[Computation of $V\0$] 
This time, the contributions to the coarse space that come from each subdomain $s \in \llbracket 1, N \rrbracket$ require solving two generalized eigenvalue problems. For the first one, $ \bL\s \bL\s^\top \by\s = \lambda \bR\s \bA  \bR\s^\top \by\s$, it is the low frequencies and their eigenvectors that must be computed. With an iterative solver, this means solving linear systems with $ \bL\s \bL\s^\top$, a cost-effective task since $\bL\s$ is triangular. \textcolor{black}{The second eigenvalue problem resembles the one for Additive Schwarz (and Neumann-Neumann), in that, naturally, the matrix that must be (pseudo-)inverted is the singular matrix $\bM\s$ (see Remark~\ref{rem:computVO}). In the context of Inexact Schwarz, the cost of computing the coarse space should not involve any local matrix factorizations. Indeed, if these factorizations are known, it probably makes more sense to exploit them also in local solvers by reverting back to Additive Schwarz or Neumann-Neumann. As an alternative, a shifting method can be applied as follows.
\begin{itemize}
\item First, an approximation $\lambda_{\max}$ of the largest eigenvalue of $\bM \s \bx\s = \lambda\s   \bL\s \bL\s^\top \bx\s $ is computed. This involves linear solves for the factorized matrix  $\bL\s \bL\s^\top$, a cost-effective task.
\item 
Then, the contribution to the coarse space coming from the second eigenvalue problem is 
\[
\mathcal Y_L(\tau^{-1}, \bM \s ,   \bL\s \bL\s^\top) = {\bL\s^\top}^{-1} \mathcal Y_H((2 \lambda_{\max} - \tau^{-1}),\,  (2 \lambda_{\max} \matid - {\bL\s}^{-1} \bM \s  {\bL\s^\top}^{-1}) ,\, \matid),
\] 
where the matrix in the eigenvalue problem is spd. (The only exception would be the unlucky case where an eigenvalue is exactly $\tau^{-1}$ because the inequality is strict in the definition of $\mathcal Y_L$ and large in the definition of $ \mathcal Y_H$ but this is easy to fix in the code.)
\end{itemize} }
\end{remark}

\subsection{Numerical results}
\label{sub:numerical}

The results in this section were obtained with the software libraries FreeFem++ \cite{MR3043640} and GNU Octave \cite{octave}. 
Let us recall that the dimension of the problem is $n = \num{7224}$. The computational domain $\Omega$ is partitioned into $N=8$ non-overlapping subdomains. Moreover, the value of the coloring constant from Definition~\ref{def:color} is  $\Ncol = 3$ and there are $n_\Gamma = 546$ degrees of freedom that belong to more than one subdomain. \textcolor{black}{It has already been observed (see \textit{e.g.}, \cite{spillane2013abstract,SPILLANE:2013:FETI_GenEO_IJNME}) that GenEO iteration counts don't depend on the number of subdomains (scalability) and the magnitude of the jump in the coefficients (as predicted by the theory). For this reason, these tests are not performed here.}

\paragraph{Krylov Subspace Method} 
The problem presented in Subsection~\ref{sub:elast} is solved by the preconditioned conjugate gradient method (PCG) with the Additive Schwarz, Neumann-Neumann, and Inexact Schwarz preconditioners. The problem is by no means a very large problem that requires state of the art parallel solvers. The purpose is to illustrate how the GenEO coarse spaces decrease the condition number and how many vectors per subdomain need to be added to the coarse space to achieve \textit{fast} convergence. The stopping criterion for PCG is always that the error $\|\bx_i - \bx\|_\bA$ be less than $10^{-9} \|\bx\|_\bA$ starting from a zero initial guess. This is not a practical stopping criterion but it has the advantage that it does not depend on the preconditioner. The numerical bounds for the spectrum of the preconditioned operators that are reported are the approximations given by the extreme Ritz values once the algorithm has converged (see \textit{e.g.}, \cite{meurant2019approximating} for details on how to implement this procedure).

\paragraph{Results for the Additive Schwarz (AS) preconditioner from Theorem~\ref{th:Ad}}
In Figure~\ref{fig:kappatau}, the condition numbers for the Additive and Hybrid preconditioners are plotted with respect to the threshold $\tau$ for both coefficient distributions and both choices of scaling. The theoretical upper bounds are also plotted and never exceeded. The theoretical bound is less sharp when $\tau$ becomes larger. As expected, preconditioning by the Hybrid preconditioner always leads to a lower condition number than preconditioning by the fully Additive variant. 

Table~\ref{tab:layers} gives a lot more information. Only the test case `with layers' is considered. In all four configurations (hybrid/additive and $\mu$-scaling/$k$-scaling), the choice $\tau = 10$ seems to offer a good compromise between the condition number (or number of iterations) and the dimension of the coarse space. With the same value of $\tau = 10$, the dimension of the coarse space with multiplicity scaling is 241 versus only 68 with $k$-scaling. This is due to the fact that the $k$-scaling already handles coefficient jumps that are across the subdomain interfaces (the ones that are already present in the `no layers' test case). With multiplicity scaling, it is the coarse space that must handle also for these jumps. To better illustrate this behaviour, Figure~\ref{fig:gevp} shows the eigenvalues of the generalized eigenvalue problem. It is clear that the even-numbered subdomains (in which $E$ is three orders of magnitude larger) must contribute many more vectors. (Recall that it is the high-frequency vectors that get selected).  This is not a failure of the GenEO method. In practice it is highly unlikely that an automatic graph partitioner would produce such a configuration. A human partitioner could produce such a configuration. If so, she would be aware of it and should choose the scaling accordingly.

\begin{figure}
\begin{center}
\includegraphics[width=0.9\textwidth]{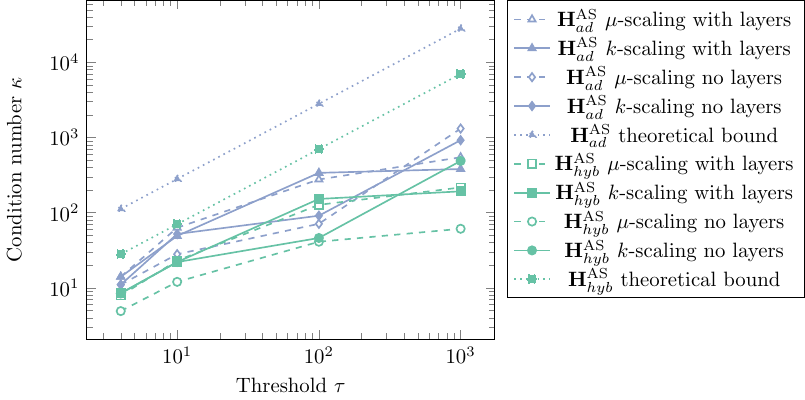}
\end{center}
\caption{Condition numbers for Additive Schwarz preconditioners: additive and hybrid; $\mu$-scaling and $k$-scaling; with and without layers; $\tau \in [4;\, 10;\, 100;\, 1000]$. All condition numbers are below the theoretical bound.}
\label{fig:kappatau}
\end{figure}

\begin{figure}
\begin{center}
\includegraphics[width=0.4\textwidth]{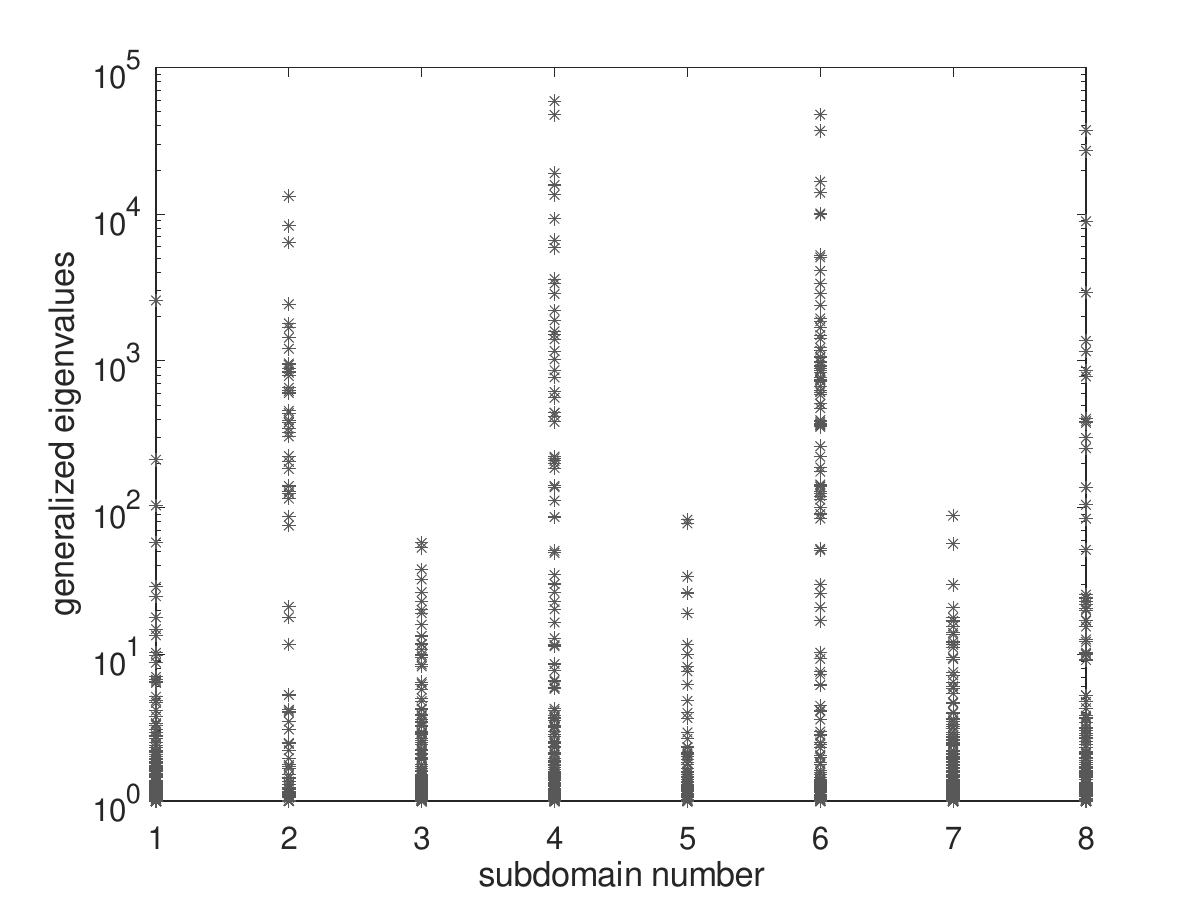} 
\includegraphics[width=0.4\textwidth]{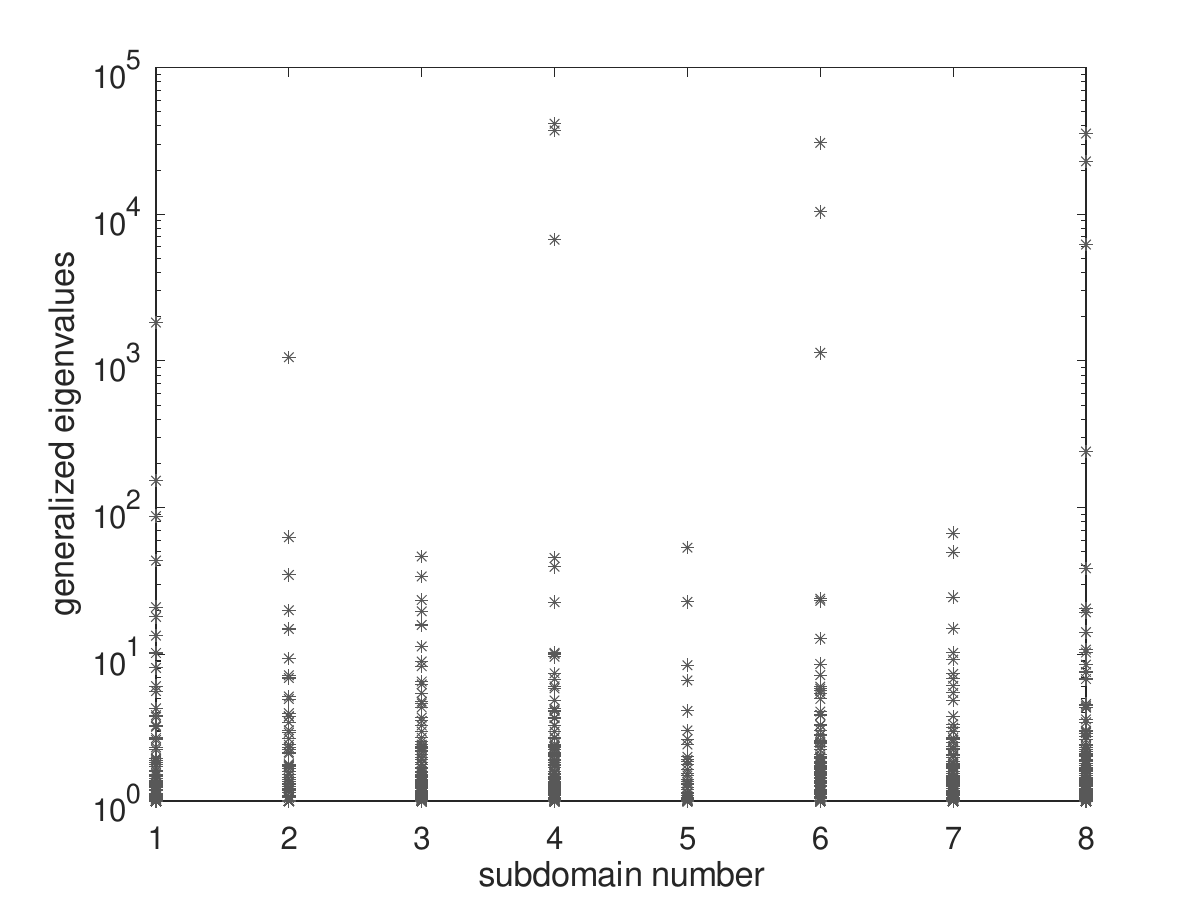} 
\end{center}
\caption{For each subdomain, solution of the generalized eigenvalue problem for computing $V\0$ in the case `with layers'. Left: $\mu$-scaling, right: $k$-scaling.}
\label{fig:gevp}
\end{figure}

\paragraph{Results also for the Neumann-Neumann (NN) and Inexact Schwarz (IS) preconditioners from Theorems~\ref{th:NN} and~\ref{th:IS}}
For lack of space these are presented in a lot less detail. Figure~\ref{fig:withwithoutkappaV0} (test cases with and without layers) shows the condition numbers with respect to the dimension of the coarse space for several values of the thresholds (that are not reported here). In this plot, the best methods are the ones which have data points closest to the origin (small condition number with a small coarse space). It appears clearly that $k$-scaling gives better results. This was previously explained. 

With $k$-scaling, the methods from most to least efficient rank as follows: Neumann-Neumann, Additive Schwarz with hybrid preconditioner, Additive Schwarz with additive preconditioner, \textcolor{black}{Inexact Schwarz with hybrid preconditioner, Inexact Schwarz with additive preconditioner.} Again, this does not tell the whole story as the cost of one iteration depends on the choice of method. With inexact Schwarz, the local solves are cheapest. With \textcolor{black}{additive variants of the preconditioners,} the coarse solve can be done in parallel to the local solves. With Neumann-Neumann, the matrices that must be handled with most care numerically (the ones that are singular) are the local solvers whereas they only appear in the generalized eigenvalue problem for the other methods.
A second word of caution about Neumann-Neumann with $k$-scaling is that the scaling matrices $\bD\s$ can be very ill-conditioned and make the whole method less efficient again. This problem is well known and independent of GenEO. All these arguments lead only to one conclusion: this data does not tell us which method is most efficient overall. The answer would in any case be \textcolor{black}{problem,} implementation and hardware dependent.

\textcolor{black}{As a final remark on the results of Figure~\ref{fig:withwithoutkappaV0}, let us comment on the Inexact Schwarz data. It can be observed that, on each Inexact Schwarz curve, there are two data points for which the dimension of the coarse space increases but the condition number does not improve (or even worsens slightly). This is because there are two parameters for the Inexact Schwarz coarse space. It is mostly only adding vectors to $Y_L(\tau^{-1},  \bM \s ,  \bL\s \bL\s^\top)$, that makes the method more efficient while adding vectors to $\mathcal Y_L(\upsilon, \bL\s \bL\s^\top , \bR\s \bA  \bR\s^\top)$ makes the coarse space grow very fast and the condition number decrease very little.}

\begin{table}
\begin{center}
{\small
\textbf{Additive Schwarz with $\mu$-scaling ($\bD\s$ from \eqref{eq:muscaling} in gevp)}

\begin{tabular}{c|c|c|c|c|c|c}
                      &     & $\lambda_{\min}$ & $\lambda_{\max}$  & $\kappa$  & It (final error)                        & $\# V\0$  ( $\min \# V\0\sups$  ; $\max \# V\0\sups$ )   \\
\hline
one-level             &     &$7.7\cdot 10^{-4}$ &  3.0            & 3875      & $>100$ ($6 \cdot 10^{-3}$) & 0         ( 0       ; 0       )  \\
\hline
$\tau = 10^{10}$& hyb & 0.003           & 3.0             & 959       & $>100$ ($1 \cdot 10^{-5}$) & 18        (  0      ; 3       )  \\
(only $\operatorname{Ker}(\bM\s)$)& ad  & 0.002           & 3.3             & 1517      & $>100$ ($3 \cdot 10^{-4}$) & 18        ( 0       ; 3 )        \\
\hline
$\tau = 1000$   & hyb & 0.014           & 3.0             & 216       & $>100$ ($1 \cdot 10^{-9}$) & 72        (  1      ; 24    )    \\
                      & ad  & 0.007           & 4.0             & 545       & $>100$ ($1 \cdot 10^{-6}$) & 72        ( 1       ; 24     )   \\
\hline
$\tau = 100$    & hyb & 0.024           & 3.0             & 127       & 92                         & 159       (  3      ; 60      )  \\
                      & ad  & 0.014           & 4.0             & 276       & $>100$ ($1 \cdot 10^{-7}$) & 159       ( 3       ; 60       ) \\
\hline
$\tau = 10$     & hyb & 0.13            & 3.0             & 23        & 42                         & 241       (  10     ; 70       ) \\
                      & ad  & 0.06            & 4.0             & 63        & 64                         & 241       ( 10      ; 70       ) \\
\hline
$\tau = 4 $     & hyb &0.37             & 3.0             & 7.9       & 23                         &  303      (  14    ; 77     )  \\
                      & ad  &0.28             & 4.0             & 14        & 31                         & 303       (  14    ; 77    )   \\
\hline
Theory ($\tau$) & hyb & $1/\tau$  &$\Ncol = 3$ & $3 \tau$ &                       &                             \\
                      & ad  & $1/(7 \tau)$           &$\Ncol+1=4$ & $28 \tau$ &                       &                      
\end{tabular}
\textbf{Additive Schwarz with $k$-scaling ($\bD\s$ from \eqref{eq:kscaling} in gevp)} 
\begin{tabular}{c|c|c|c|c|c|c}
                      &     & $\lambda_{\min}$ & $\lambda_{\max}$  & $\kappa$  & It (final error)                        & $\# V\0$ ($\min \# V\0\sups$; $\max \# V\0\sups$)    \\
\hline
one-level             &     &$7.7\cdot 10^{-4}$ &  3.0            & 3875      & $>100$ ($6 \cdot 10^{-3}$) & 0         ( 0       ; 0    )     \\
\hline
$\tau = 10^{10}$& hyb & 0.0030           & 3.0             & 1003       & $>100$ ($2 \cdot 10^{-5}$) & 18        (  0      ; 3     )    \\
(only $\operatorname{Ker}(\bM\s)$)& ad  & 0.0025          & 3.2             & 1271      & $>100$ ($2 \cdot 10^{-4}$) & 18        ( 0       ; 3 )        \\
\hline
$\tau = 1000$   & hyb & 0.016           & 3.0             & 192       & 98                          & 29        (  1      ; 6 )       \\
                      & ad  & 0.0087           & 3.3             & 380       & $>100$ ($3 \cdot 10^{-7}$) & 29        ( 1       ; 6 )       \\
\hline
$\tau = 100$    & hyb & 0.02           & 3.0             & 152       & 93                         & 31       (  1      ; 7    )    \\
                      & ad  & 0.098           & 3.3            & 338       & $>100$ ($2 \cdot 10^{-7}$) & 31       ( 1       ; 7    )    \\
\hline
$\tau = 10$     & hyb & 0.13            & 3.0             & 22        & 43                         & 68       (  5     ; 13    )    \\
                      & ad  & 0.069           & 3.37            & 49        & 63                         & 68       ( 5      ; 13   )     \\
\hline
$\tau = 4 $     & hyb &0.35             & 3.0             & 8.5       & 26                         &  118      (  8    ; 20    )   \\
                      & ad  &0.25             & 3.4             & 14        & 34                         &  118      (  8    ; 20    )   \\
\hline
Theory ($\tau$) & hyb & $1/\tau$  &$\Ncol = 3$ & $3 \tau$ &                       &                              \\
                      & ad  & $1/(7 \tau)$           &$\Ncol+1=4$ & $28 \tau$ &                       &                               
\end{tabular}
}
\end{center}
\caption{Test case `with layers' - All additive Schwarz methods - $\lambda_{\min}$ and $\lambda_{\max}$: extreme eigenvalues, $\kappa$:  condition number, It: iteration count (with relative error at iteration 100 in parenthesis if the method has not converged),  $\# V\0$ : dimension of the coarse space, $\min \# V\0\sups$: number of coarse vectors contributed by the subdomain that contributes the fewest vectors, $\max \# V\0\sups$: number of coarse vectors contributed by the subdomain that contributes the most eigenvectors, gevp: generalized eigenvalue problem. The one-level method does not satisfy any theoretical bound for $\lambda_{\min}$.}
\label{tab:layers} 
\end{table}

\begin{figure}
\begin{center}
\textbf{`No layers'}\\
\includegraphics[width=\textwidth]{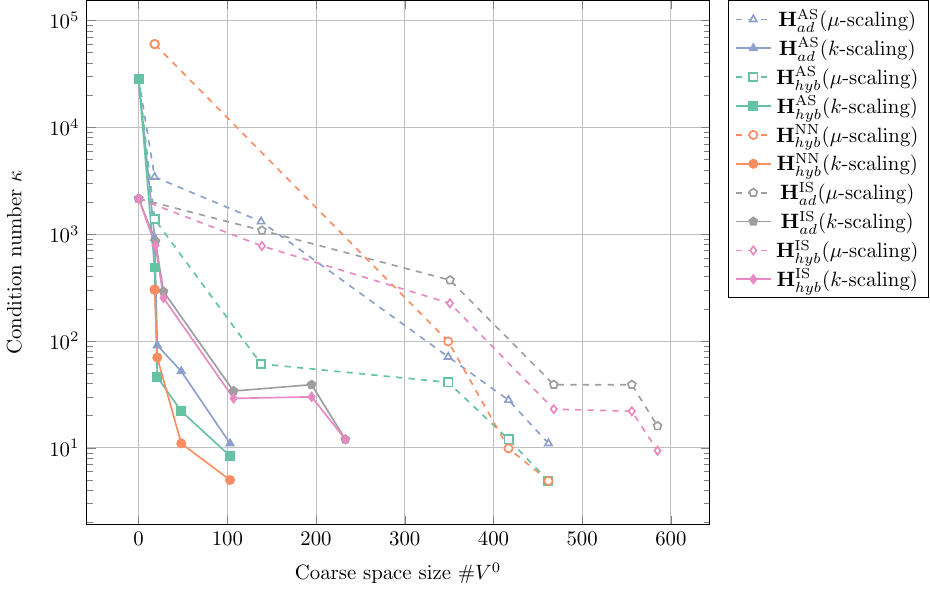}
\\
\bigskip
\textbf{`With layers'}\\
\includegraphics[width=\textwidth]{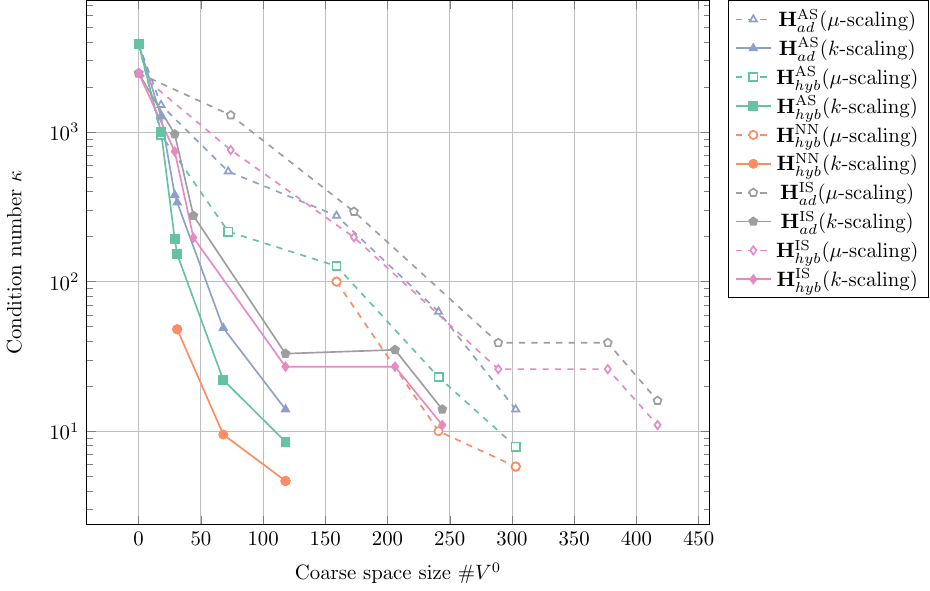}
\end{center}
\caption{Efficiency of all methods (Condition number versus coarse space dimension). Top: the test case without layers in $E$. Bottom: test case with layers in $E$.}
\label{fig:withwithoutkappaV0}
\end{figure}

\begin{table}
\textcolor{black}{
\begin{center}
\textbf{Subdomains computed by Metis (same as in previous tests)}\\
\begin{tabular}{c|c|c|c|c|c|c|c}
overlap &  $\# \Omega\s$ & scaling &  AS hyb & AS ad & NN hyb & IS hyb & IS ad \\
\hline
0 & 972  &   $\mu$-scaling    & 42 (241) & 64 (241) & 30 (241) & 50 (289) & 59 (289) \\
\hline
1 & 1083 &   $\mu$-scaling    & 37 (90) & 55 (90) & 42 (90) & 56 (132) & 67 (132) \\ 
\hline
2 & 1199 &   $\mu$-scaling    & 34 (81) & 51 (81) & 45 (81) & 58 (126) & 66 (126) \\
\hline
\hline
0 & 972  &   $k$-scaling    & 43 (68)  & 63 (68)  & 42 (90)  & 50 (118)  & 56 (118) \\
\end{tabular}
\\
\bigskip
\textbf{Regular Subdomains (does not match the coefficient distribution)}\\
Note: This is the only test with the regular partition\\
\begin{tabular}{c|c|c|c|c|c|c|c}
overlap &  $\# \Omega\s$ & scaling &  AS hyb & AS ad & NN hyb & IS hyb & IS ad \\
\hline
0 & 957  &   $\mu$-scaling    & 46 (63) & 65 (63) & 32 (63) & 51 (115) & 58 (115) \\ 
\hline
1 & 1068  &   $\mu$-scaling    & 48 (55) & 66 (55) & 50 (55) & 60 (107) & 66 (107) \\ 
\hline
2 & 1184 &   $\mu$-scaling    & 43 (51) & 60 (51) & 53 (51) & 64 (101) & 76 (101) 
\end{tabular}
\caption{Test case `with layers' in $E$. Influence of the overlap. overlap: number of layers of overlap added to each subdomain. $\# \Omega\s$: average number of degrees of freedom per subdomain. In each case, the iteration count is reported and, in parenthesis, the dimension of the coarse space (\textit{i.e.} It ($\# V\0$) with notation from the previous table). The first table is for the partition computed by Metis (same as all previous test cases). Because the Metis partition is connected to the coefficient distribution, we solve the same problem with regular subdomains. For regular subdomains and our problem, both scalings are the same.}
\label{tab:ovr}
\end{center}
}
\end{table}

\textcolor{black}{
\paragraph{Overlapping subdomains}
In the context of Additive Schwarz preconditioners, it is common practice to consider subdomains that share more overlap than just their common interfaces. This is not the case for Neumann-Neumann methods. Table~\ref{tab:ovr} studies the influence of the overlap on the efficiency of the preconditioners in terms of iteration count versus coarse space dimension. The overlap is constructed by adding to each subdomain either 1 or 2 layers of elements in every direction. With overlap, applying $k$-scaling, as defined in formula~\eqref{eq:kscaling}, no longer produces a partition of unity so only multiplicity scaling is considered.  The thresholds for the coarse spaces are set to $\tau = 10$ for Additive Schwarz, $\tau = 0.1$ for Neumann-Neumann, and $(\tau; \, \upsilon) = (10;\, 0.1)$ for Inexact Schwarz. The coefficient distribution is the case `with layers' (Figure~\ref{fig:geom} -- right). A word of caution is that the condition number bounds presented at the beginning of the section must be updated to include the correct value of $\Nprime$ if overlap is considered. 
\\
The first set of results in Table~\ref{tab:ovr} considers the same problem as previously. It is observed that adding one layer of overlap has a drastic effect on the coarse space dimension: it decreases from 289 to 132 for Inexact Schwarz and from 241 to 90 for the other methods. The effect on the iteration count is positive for Additive Schwarz and negative for the other methods. The trend is the same when passing from 1 layer of overlap to 2 but with a much less significant decrease in the coarse space dimension. It is important to realize that this test case is very particular because, without overlap, there are jumps in the coefficient across the interfaces that are not compensated by $\mu$-scaling. This is the reason why the decrease in coarse space dimension is so dramatic when passing from $0$ to $1$ layer of overlap: the interfaces no longer match the jumps. In order to confirm this, a final test is run with regular subdomains ($8$ squares of dimension $1/2 \times 1/2$ before overlap is added). Now, the partition into subdomains and the coefficient distribution no longer match. The trends are the same but much less pronounced as expected. In this particular configuration and without overlap, $k$-scaling and $\mu$-scaling are identical. To make this analysis complete, the average dimension of the coarse problems as been added to Table~\ref{tab:ovr} as a reminder that it grows with the overlap, as does the cost of communication.       
\\
It does not appear clearly to the author that adding overlap is always beneficial, particularly for methods other than Additive Schwarz. On the contrary, improving the scaling has always proved to be a good idea.
}

\section{Conclusion}

\textcolor{black}{
GenEO coarse spaces have been introduced for all domain decomposition methods in the abstract Schwarz framework under clearly stated assumptions. By solving one or two generalized eigenvalue problems in each subdomain, it is possible to construct two-level methods for which the eigenvalues of the preconditioned operator are bounded in a chosen interval. Proofs of these bounds are given for three variants of the preconditioner: projected, hybrid and, when possible, additive. As a by-product of the analysis, two core results of the abstract Schwarz framework (commonly referred to as the stable splitting property and the stability of the local solver) have been extended to singular local problems and projected local subspace. Finally, the methodology has been applied to two of the usual candidates for domain decomposition with a GenEO coarse space (Additive Schwarz and Neumann-Neumann) as well as one new one (Inexact Schwarz). Their performances are analyzed and compared on a linear elasticity problem discretized by $\mathbb P_1$ finite elements. It is advocated that particular attention should be payed to the choice of scaling in the partition of unity, and that adding overlap beyond the shared interfaces is not an obligation for performance.
}

\bibliographystyle{abbrv}
\bibliography{AbstractGenEO}
\end{document}